\documentclass[11pt]{article}
\usepackage{amssymb,latexsym,amsmath,amsbsy,amsthm,amsxtra,amsgen,graphicx,makeidx,dsfont,sidecap}
\newfont{\msbm}{msbm10 at 11pt}

\newcommand {\Z} {\mbox{\msbm Z}}

\newcommand {\N} {\mbox{\msbm N}}
\newcommand {\1} {\mathds{1}}

\newfont{\msbmsm}{msbm10 at 8pt}

\newtheorem{Theo}{Theorem}
\newtheorem{Lemma}[Theo]{Lemma}
\newtheorem{Cor}[Theo]{Corollary}
\newtheorem{Prop}[Theo]{Proposition}

\def\eps{\varepsilon}
\def\Var{\textup{Var}}

\begin{document}
\title{The Nested Kingman Coalescent: \\ Speed of Coming Down from Infinity}
\author{Airam Blancas Ben\'itez\thanks{Goethe Universit\"at Frankfurt, email: blancas@math.uni-frankfurt.de}, \: Tim Rogers\thanks{University of Bath, email: t.c.rogers@bath.ac.uk}, \: Jason Schweinsberg\thanks{University of California San Diego, email: jschwein@math.ucsd.edu}, \: Arno Siri-J\'egousse\thanks{Universidad Nacional Aut\'onoma de M\'exico, email: arno@sigma.iimas.unam.mx}}
\maketitle

\vspace{-.3in}
\begin{abstract}
The nested Kingman coalescent describes the ancestral tree of a population undergoing neutral evolution at the level of individuals and at the level of species, simultaneously. We study the speed at which the number of lineages descends from infinity in this hierarchical coalescent process and prove the existence of an early-time phase during which the number of lineages at time $t$ decays as $ 2\gamma/ct^2$, where $c$ is the ratio of the coalescence rates at the individual and species levels, and the constant $\gamma\approx 3.45$ is derived from a recursive distributional equation for the number of lineages contained within a species at a typical time. 
\end{abstract}

\section{Introduction}
Kingman's coalescent \cite{kingman82} lies at the centre of modern mathematical population genetics. It is a simple probabilistic model describing the ancestral tree of a population undergoing neutral evolution, which has been shown to apply to a wide variety of population dynamical models \cite{mohle00}, and gives rise to the hugely important Ewens sampling formula \cite{ewens72} for the expected genetic variation within a population. Work on Kingman's coalescent and its variants has fueled a wealth of developments in the probability literature, summarised succinctly in \cite{berestycki09r}. 

A key result of this theory is that Kingman's coalescent \emph{comes down from infinity}, meaning coalescence occurs so quickly that even when the process is started with an infinite number of lineages, only finitely many survive after any positive time. It is in fact possible to be more precise and state the speed of this descent from infinity. Let $K_n(t)$ denote the number of lineages surviving to time $t$ in the Kingman coalescent initialized on a population of size $n$.  Theorem 1 of \cite{berestycki10} (see also \cite{aldous99}) states that taking $n\to\infty$ and then $t\to0$ we have the almost sure convergence $tK_n(t)\to2$.
Thus, for small times the number of surviving lineages in the Kingman coalescent decays as $1/t$. This result is important to the population genetics community as it characterizes the expected shape of the lineages through time (LTT) plot \cite{harvey94,mooers97}, a popular technique for analyzing phylogenetic trees reconstructed from genetic data. 
The speed of descent from infinity has also been studied for coalescents with multiple mergers in \cite{berestycki10} and for more general birth and death processes in \cite{bans}.

From the perspective of applications to genetics, a limitation of Kingman's coalescent is that it describes only the historical coalescence of lineages within a species, and can not at the same time account for macroevolutionary events occurring between species. The problem of how the gene tree is embedded inside the species tree has been one of the central research questions of population genetics for some time now (see, e.g. \cite{maddison97,szollosi14}), and the issue of how to draw the distinction between intra-  and inter-specific genetic variation is an important and contested one \cite{rossberg13,morgan14,rossberg14}.

In this article we address this defect in the theory by computing the speed of descent from infinity in a nested (hierarchical) coalescent process which models both the species tree and the embedded gene tree as a Kingman coalescent, with the latter constrained to be embedded in the former -- see Figure~1 for an illustration. We prove that this model exhibits an early-time period in which the number of lineages decays as $1/t^2$; much faster than Kingman's coalescent. This result is potentially important for the environmental metagenomics community, where differentiating between inter- and intra-specific genetic variation is a key step in quantifying biodiversity (see e.g. \cite{creer2010}). Empirical verification of a $1/t^2$ scaling in the LTT plot of an experimentally reconstructed phylogeny would suggest, according to our results, that the gene tree and species tree are evolving on the same time scale, greatly complicating this task.

The article is organised as follows. In the remainder of this section we give the formal definition of our process (and its population dynamical dual), and state our main theorem. Section \ref{sec:kingman} develops several results for the standard Kingman coalescent to do with the rate of decrease of the number of lineages, and the asymptotic independence of branches in the ancestral tree. These results are needed for our investigation since, in the nested model, both the species tree and the within-species gene trees (before and between species merger events) are described by Kingman's coalescent. Section \ref{sec:nested} brings together the results of Section \ref{sec:kingman} to deduce a recursion relation between species merger events in the nested coalescent and thence prove our main theorem. 

\subsection{Definition of the model}

We consider the following nested coalescent model. We begin with a sample of $n$ individuals from each of $s$ species (including the possibility that one or both of $n$ and $s$ is infinite). Each pair of individuals within a species merges at rate one; also, each pair of species merges at rate $c > 0$.  More formally, this process is a continuous-time Markov chain
taking its values in the set of labeled partitions of $\{(m,k) \in \Z \times \Z: 1 \leq m \leq n, 1 \leq k \leq s\}$, in which each block of the partition is labeled with one of the integers $1, \dots, s$.  At time zero, the partition consists of $ns$ singleton blocks, and the block $(m, k)$ is labeled by the integer $k$.  Two types of transition are possible:
\begin{description}
\item[Lineage mergers] Any pair of blocks with the same label may merge into a single block with that label, with rate $1$.
\item[Species mergers] For any pair of currently surviving labels $i < j$, all blocks with label $j$ have their label changed to $i$, with rate $c$.
\end{description}
We refer to this model as the nested Kingman coalescent because, both at the individual and species level, the merging follows the rule of the classical Kingman coalescent \cite{kingman82}. This model has appeared before in the literature in \cite{dawson18}.
This model can be alternatively seen as a coalescent process with values in the set of bivariate nested partitions.
It is actually an example of simple nested coalescents as defined in \cite{blancas18}.
In this reference, a criterion is provided to determine whether nested coalescents come down from infinity or not.
However, to our knowledge the speed of descent from infinity has not been computed previously. 
\begin{SCfigure}\label{drawing}
\includegraphics[width=0.35\textwidth]{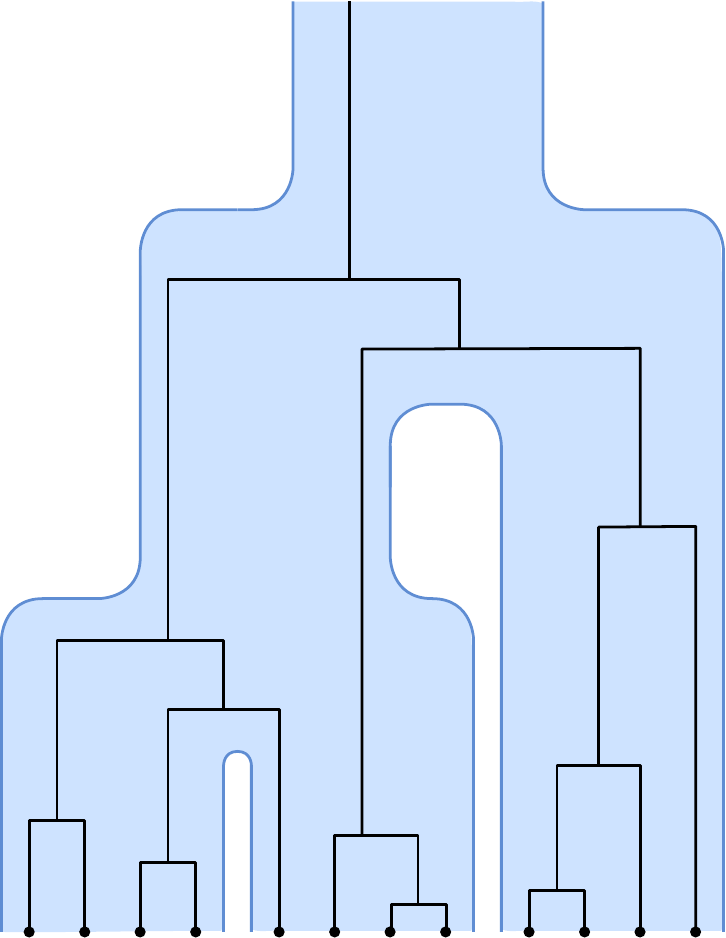}
\caption{Illustration of the nested Kingman coalescent starting with $s=3$ species and $n=4$ lineages per species. Black lines show a possible ancestral tree for the sampled individuals, with lineage mergers constrained to lie within the species tree (shown behind in pale blue). The species mergers are described by a Kingman coalescent with rate $c$, while the within species lineage mergers form a Kingman coalescent with rate $1$.}
\end{SCfigure}

The nested Kingman coalescent describes the genealogy in the following population model.  Consider a population divided into $s$ species, each composed of $N$ individuals. Within each species, the population evolves according to the classical Moran model \cite{moran58}. That is, each individual lives for an exponentially distributed time with mean $1$; when an individual dies, a new individual is born, and one of the $N$ individuals of the species is chosen at random to be the parent of the new individual.  To model the formation of new species, we also suppose that each species becomes extinct after an exponentially distributed time with rate $c(s-1)/N$, at which time all members of the species simultaneously die.  At that time, $N$ new individuals are born, forming a new species.  One of the $s$ species is chosen at random, and each member of that species gives birth to one member of the new species.  After scaling time by $N/2$, the genealogy of a sample consisting of $n$ individuals from each species converges to the nested Kingman coalescent in the limit as $N \rightarrow \infty$ because the large population size ensures that with probability tending to one as $N \rightarrow \infty$, the sampled ancestral lines will not merge at the times when new species form.
Similar to the standard Kingman coalescent, we expect that the nested Kingman coalescent will also appear as the asymptotic form of various other similar population models under suitable limits.  However, this is not the topic of our present study. 

\subsection{Main Results}
At time $t$ we write $S(t)$ for the number of species, and $N(t)$ for the total number of blocks (i.e. extant ancestral lines) across all species. Informally, our main result is that, if the initial number of species is large, then there is a period of time during the early evolution of the process in which $N(t)$ decays as $1/t^2$. Since the number of blocks in the standard Kingman coalescent decays as $1/t$, one can understand the $1/t^2$ decay observed in the nested process as a consequence of mergers occurring on both scales (individuals within a species, and whole species mergers) simultaneously. 

To state this claim precisely, it is necessary to consider a sequence of processes. For $j\in\mathbb{N}$,
consider
an instance of the nested Kingman coalescent in which the initial number of species is $s_j$ and the number of individuals sampled from each species is $n_j$ (which, for simplicity, is assumed to be the same for each species). We allow the cases in which $s_j = \infty$ or $n_j = \infty$.
Using the notation $a_j \ll b_j$ to mean $\lim_{j \rightarrow \infty} a_j/b_j = 0$, $\rightarrow_{\mathrm{p}}$ to denote convergence in probability, and $=_{\mathrm{d}}$ to denote equality of distributions, our main result is expressed in the following theorem. 

\begin{Theo}\label{Speed}
Suppose $1/s_j \ll t_j \ll 1$, and $1/\sqrt{n_j s_j}\ll t_j$. Then 
$$t_j^2 N(t_j) \rightarrow_{\mathrm{p}} \frac{2\gamma}{c} \hspace{.2in}\mbox{as } j \rightarrow \infty.$$  Here $\gamma$ is the mean of the uniquely determined random variable $W$ that takes values in $[2,\infty)$ and obeys the recursive distributional equation
\begin{equation}\label{RDE}
W=_{\mathrm{d}} \frac{2}{1 - U\big(1 - \frac{2}{W_1 + W_2}\big)},
\end{equation}
where $U$ has a uniform distribution on $[0,1]$, $W_1$ and $W_2$ have the same distribution as $W$, and the random variables $U$, $W_1$, and $W_2$ are independent.
\end{Theo}

When $s_j \equiv \infty$, Theorem \ref{Speed} implies that $$t^2 N(t) \rightarrow_{\mathrm{p}} \frac{2 \gamma}{c} \hspace{.2in}\mbox{as } t \rightarrow 0.$$  Therefore, in this case Theorem~\ref{Speed} gives the speed at which $(N(t), t \geq 0)$ descends from infinity. 
Note that the hypotheses of Theorem \ref{Speed} require $s_j \to \infty$, but not necessarily that $n_j \to \infty$.  For example, the case $n_j \equiv 1$, which corresponds to sampling one individual of each species, is included.  When $n_j$ equals some fixed constant $n$ for all $j$, Theorem \ref{Speed} implies that, for any fixed $t>0$ and $\beta\in(0,1/2)$, $$\frac{N(ts^{-\beta})}{s^{2\beta}} \rightarrow_{\mathrm{p}} \frac{2 \gamma}{ct^2} \hspace{.2in}\mbox{as } s \rightarrow \infty.$$ 
This scaling can be compared with non-nested models such as Beta-coalescents (see Theorem 4.4 of \cite{DFSY}).

In the case that the initial number of lineages per species vastly exceeds the number of species ($n_j\gg s_j$), the period of $1/t^2$ scaling implied by Theorem \ref{Speed} is preceded by an earlier phase dominated entirely by within-species coalescence. There, the usual $1/t$ scaling is recovered, as we make explicit in the following proposition.

\begin{Prop}\label{newprop}
Suppose $s_j \rightarrow \infty$ and $1/n_j \ll t_j \ll 1/s_j$.  Then $$\frac{t_j N(t_j)}{s_j} \rightarrow_{\mathrm{p}}2 \hspace{.2in}\mbox{as } j \rightarrow \infty.$$
\end{Prop}

The proof of Proposition \ref{newprop} is given in Section \ref{newpropproof}, and a numerical example showing both $1/t$ and $1/t^2$ phases is shown in Figure \ref{bothphases}. Different techniques have to be developed when $t$ is of the order $1/s_j$, see \cite{schertzer18}.

\subsection{Heuristics and simulations}
Before presenting our proofs, it is instructive to consider a simple mean-field heuristic for the time-evolution of the process.  For the purposes of this discussion, we will focus on the case when $n \gg s \gg 1$.  First note that the process $(S(t), t \geq 0)$ has the same law as the number of blocks in Kingman's coalescent (with time scaled by a factor of $c$).  Therefore (following \cite{aldous99}), for small times we can approximate $S(t)$ by the solution to the differential equation $$\frac{d}{dt} S(t) \approx - \frac{c S(t)^2}{2}, \hspace{.5in} S(0) = s.$$  It follows that when $t \ll 1$, we have
\begin{equation}\label{Seq}
S(t) \approx \frac{2}{ct + \frac{2}{s}} \,.
\end{equation}
We have $N(t) = N_1(t) + \dots + N_{S(t)}(t)$, where $N_i(t)$ denotes the number of lineages belonging to the $i$th of the $S(t)$ species at time $t$.  When $t \ll 1/s$, we see from (\ref{Seq}) that $S(t) \approx s$, which means very few species mergers have occurred.  Within each species, the lineages are merging according to Kingman's coalescent.  Therefore, during this period, $N_i(t)$ can be approximated by the solution to the differential equation
$$\frac{d}{dt} N_i(t) \approx - \frac{N_i(t)^2}{2}, \hspace{.5in}N_i(0) = n.$$
It follows that
\begin{equation}\label{Niapprox}
N_i(t) \approx \frac{2}{t + 2/n}
\end{equation}
for $t \ll 1/s$ and, in particular, $N_i(t) \approx 2/t$ when $1/n \ll t \ll 1/s$.  Consequently, we should have $N(t) \approx 2s/t$ when $1/n \ll t \ll 1/s$, which is consistent with Proposition \ref{newprop}.

Note, however, that the number of lineages belonging to a given species will jump upwards when two species merge into one.  Consequently, once species mergers start to occur around times of order $1/s$, we can no longer approximate the quantities $N_i(t)$ by solutions to a differential equation.  Indeed, these random variables will no longer be well approximated by their expectation, due to the randomness resulting from the timing of the species mergers.  Instead, we will argue that when $1/s \ll t \ll 1$, the distribution of $N_i(t)$ is well approximated by the distribution of $W/t$, where $W$ satisfies the recursive distributional equation (\ref{RDE}).  The Law of Large Numbers then suggests the approximation
\begin{equation}\label{Napprox}
N(t) \approx S(t) \cdot \frac{E[W]}{t} \approx \frac{2}{ct} \cdot \frac{\gamma}{t} = \frac{2 \gamma}{c t^2},
\end{equation}
which matches the result of Theorem \ref{Speed}.  Therefore, we see the possibility of both $1/t$ and $1/t^2$ behaviour, depending on the parameters.
Figure \ref{bothphases} shows an example simulation of the nested Kingman coalescent in which both phases of decay are visible.
Figure \ref{nestedc} shows several example simulations for different values of $c$, compared to the asymptotic result (\ref{Napprox}).

\begin{SCfigure}
\includegraphics[width=0.5\textwidth]{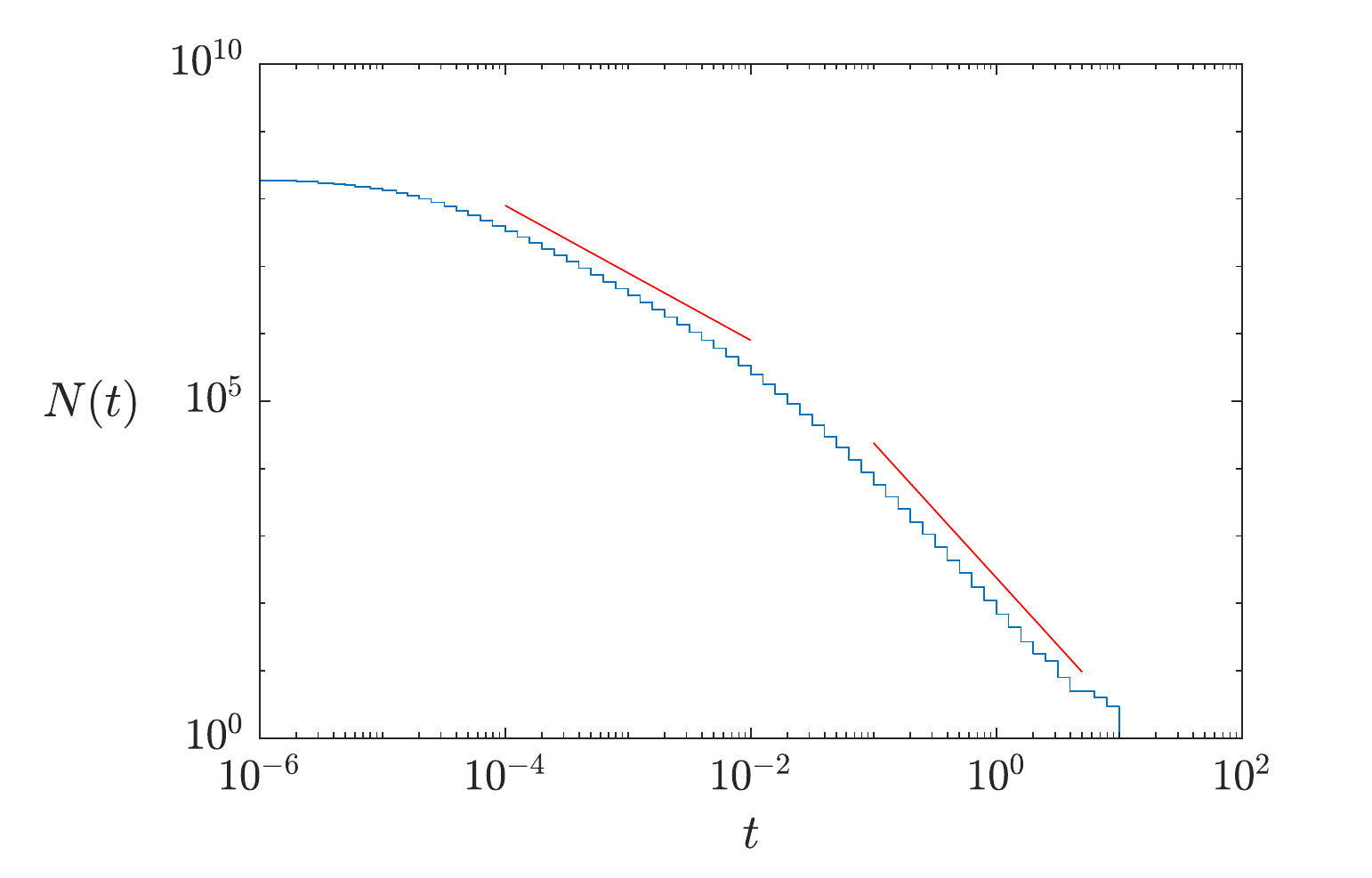}\label{bothphases}
\caption{Simulation of the nested Kingman coalescent in the case $n\gg s \gg 1$. Here the $1/t^2$ phase is preceded by a period of $1/t$ decay, corresponding to the coalescence occurring within species, but before the species coalescence events kick in. The blue line shows the result of a single simulation with $s=2000, n=100000, c=0.1$, the red lines indicate slopes of $-1$ and $-2$ to illustrate the different scaling regimes.}
\end{SCfigure}

\begin{SCfigure}\label{nestedc}
\includegraphics[width=0.5\textwidth]{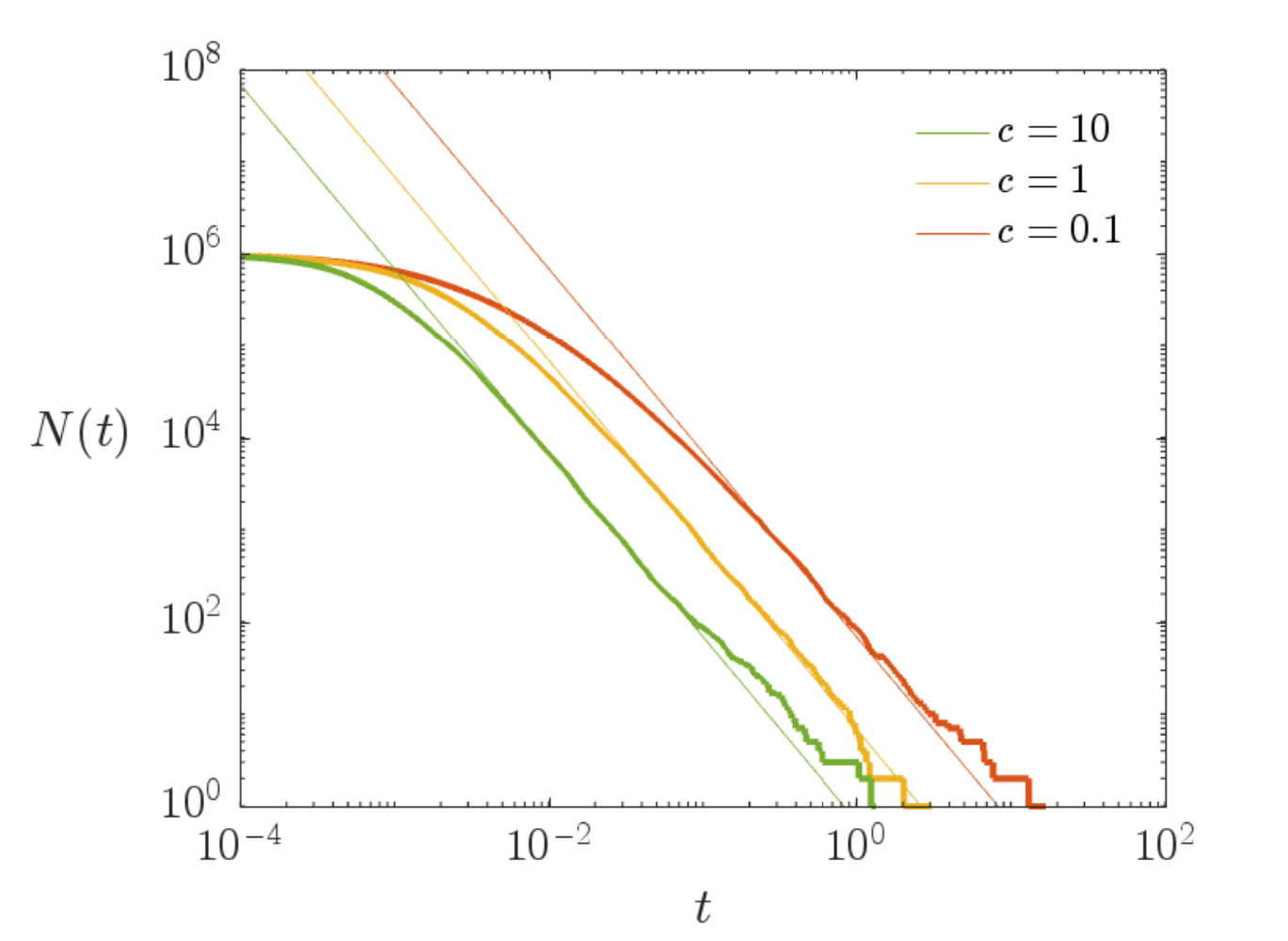}
\caption{Simulations of the nested Kingman coalescent for various values of $c$, compared to the asymptotic form $N(t)\sim 2\gamma/ct^2$ given in Theorem \ref{Speed}. In each case the simulation was started with $s=1000$ species each with $n=1000$ members.}
\end{SCfigure}

To understand the recursive distributional equation \eqref{RDE}, we consider choosing at random one of the $S(t)$ species at time $t$.
We then look for the last species merger in the species subtree rooted at this individual at time $t$.  It is well-known that this species merger happens at time $Ut$, where the distribution of $U$ is approximately uniform on $[0, 1]$, as we will explain in more detail in section~\ref{KingmanYule} below.  Then, at time $Ut$, we merge two species with $W_1/Ut$ and $W_2/Ut$ individual lineages respectively, where $W_1$ and $W_2$ are independent and have the same distribution as $W$.  Because the resulting $(W_1 + W_2)/Ut$ lineages then merge as in Kingman's coalescent for the remaining $(1-U)t$ time, the number of lineages left at time $t$ is given by the right-hand side of (\ref{Niapprox}) with $(W_1 + W_2)/Ut$ in place of $n$ and $(1-U) t$ in place of $t$.  That is, we get the approximation $$N_i(t) \approx \frac{2}{(1 - U)t + \frac{2U t}{W_1 + W_2}} = \frac{1}{t} \cdot \frac{2}{1 - U\big(1 - \frac{2}{W_1 + W_2}\big)} \,.$$  Writing $N_i(t) \approx W/t$ leads to (\ref{RDE}).

Straightforward bounds on the constant $\gamma$ can be obtained based on the conditional expectation 
\begin{equation}
E[W|W_1,W_2]=\int_0^1 \frac{2}{1 - u(1 - 2/(W_1 + W_2))}  \,\textrm{d}u =\frac{2}{\frac{2}{W_1+W_2}-1}\log\left(\frac{2}{W_1+W_2}\right)\,.
\label{EWcond}
\end{equation}
On the one hand, we know that $W_1,W_2>2$, and hence we obtain 
\begin{equation}
\gamma>E[W|W_1=W_2=2]=4\log(2)\approx 2.7726 \,.
\end{equation}
On the other hand, the right-hand-side of (\ref{EWcond}) is a concave function of the sum $W_1+W_2$, which has expectation $2\gamma$, hence by Jensen's inequality we must have
\begin{equation}
\gamma<\frac{2}{1/\gamma-1}\log\left(\frac{1}{\gamma}\right)\,.
\end{equation}
Solving at equality we obtain the upper bound
\begin{equation}
\gamma<-2W_{-1}(-1/2\sqrt{e})\approx 3.5129\,,
\end{equation}
where $W_{-1}$ denotes the lower branch of the Lambert W function. 

We have also simulated from the distribution of $W$ by constructing binary trees of height 12 and using the ``recursive tree process" discussed in more detail in section 2.3 of \cite{aldous05}.  Two random variables $W_L$ and $W_U$ were obtained from each run of the procedure.  The random variable $W_L$ was obtained by starting with values of $2$ at the leaf notes, while $W_U$ was obtained by starting with $\infty$ at the leaf nodes.  The same uniform random variables were used to obtain $W_L$ and $W_U$, which ensured that $W_L \leq W_U$.  Furthermore, $W$ stochastically dominates $W_L$ and is stochastically dominated by $W_U$.  This procedure was repeated 10,000,000 times.  The values for $W_L$ had a mean of 3.4466, and the values for $W_U$ had a mean of 3.4467.  The standard error of these estimates was .0009, which means we can be 95 percent confident that
$3.4457 < \gamma < 3.4476.$

\section{Results on Kingman's coalescent}\label{sec:kingman}

\subsection{Estimates on the number of blocks}

Let $(\Pi_{\infty}(t), t \geq 0)$ be Kingman's coalescent \cite{kingman82}, which is a stochastic process taking its values in the set of partitions of $\N$, and let $(\Pi_n(t), t \geq 0)$ be the restriction of $(\Pi_{\infty}(t), t \geq 0)$ to $\{1, \dots, n\}$.  Recall Kingman's coalescent is defined by the property that, for each $n$, the process $(\Pi_n(t), t \geq 0)$ is a continuous-time Markov chain such that each transition that involves two blocks of the partition merging together happens at rate one, and no other transitions are possible.  Let $K_n(t)$ denote the number of blocks of the partition $\Pi_n(t)$, and let $K_{\infty}(t)$ denote the number of blocks of $\Pi_{\infty}(t)$.  Theorem 1 of \cite{berestycki10} (see also \cite{aldous99}) states that 
\begin{equation}\label{cdi}
\lim_{t \rightarrow 0} t K_{\infty}(t) = 2 \hspace{.2in} \textup{a.s.}
\end{equation}
Theorem 2 of \cite{berestycki10} implies that for all $d \geq 1$,
\begin{equation}\label{kingmom}
\lim_{t \rightarrow 0} E \bigg[ \bigg| \frac{t K_{\infty}(t)}{2} - 1 \bigg|^d \bigg] = 0.
\end{equation}
Our next result provides a first moment estimate for the coalescent started with $n$ blocks.

\begin{Lemma}\label{kinglem}
Let $\delta > 0$.  There exists a positive number $t_0$ and a positive integer $M$, both depending on $\delta$, such that for all $t \leq t_0$ and $n \geq M$, we have
$$E \bigg[ \bigg| K_n(t) - \frac{2}{t + 2/n} \bigg| \bigg] \leq \frac{\delta}{t + 2/n}.$$
\end{Lemma}

\begin{proof} 
Let $0 < \eps < 1$.  By (\ref{kingmom}) with $d = 1$, there exists $t_1 > 0$, depending on $\eps$, such that if $t \leq t_1$ then
\begin{equation}\label{bblexp}
E \bigg[ \bigg| K_{\infty}(t) - \frac{2}{t} \bigg| \bigg] \leq \frac{\eps}{t}.
\end{equation}
Also, (\ref{cdi}) implies that for sufficiently large $n$,
\begin{equation}\label{bblprob}
P \bigg( K_{\infty} \Big( \frac{2(1 + \eps)}{n} \Big) \leq n \leq K_{\infty} \Big( \frac{2(1 - \eps)}{n} \Big) \bigg) > 1 - \eps.
\end{equation}

The random variable $K_n(t)$ is stochastically bounded from below by a random variable $Y_1$, which equals $K_{\infty}(t + 2(1 + \eps)/n)$ on the event that $K_{\infty}(2(1 + \eps)/n) \leq n$ and zero otherwise.  Then, denoting the positive and negative parts of a random variable $X$ by $X^+$ and $X^-$ and using (\ref{bblexp}) and (\ref{bblprob}), we get that if $t + 2(1 + \eps)/n \leq t_1$ and $n$ is sufficiently large, then
\begin{align*}
E \bigg[ \bigg(K_n(t) - \frac{2}{t + 2/n} \bigg)^- \bigg] &\leq E \bigg[ \bigg(Y_1 - \frac{2}{t + 2/n} \bigg)^- \bigg] \\
&\leq E \bigg[ \bigg| K_{\infty} \Big( t + \frac{2(1 + \eps)}{n} \Big) - \frac{2}{t + 2/n} \bigg| \bigg] \\
&\hspace{1in} + \frac{2}{t + 2/n} P \bigg(K_{\infty} \Big(t + \frac{2(1 + \eps)}{n} \Big) > n \bigg) \\
&\leq \frac{\eps}{t + 2(1 + \eps)/n} + \bigg( \frac{2}{t + 2/n} - \frac{2}{t + 2(1 + \eps)/n} \bigg) + \frac{2 \eps}{t + 2/n} \\
&\leq \frac{3 \eps}{t + 2/n} + \frac{2}{(t + 2/n)^2} \cdot \frac{2 \eps}{n} \\
&\leq \frac{5 \eps}{t + 2/n}.
\end{align*}

Let $(K'_{\infty}(t), t \geq 0)$ be an independent copy of the process $(K_{\infty}(t), t \geq 0)$.
The random variable $K_n(t)$ is stochastically bounded from above by a random variable $Y_2$ that equals $K_{\infty}(t + 2(1-\eps)/n)$ on the event that $K_{\infty}(2(1-\eps)/n) \geq n$.  On the event that $K_{\infty}(2(1 - \eps)/n) \leq n$, we set $Y_2 = n$ if $n \leq (2 + \eps)/t$ and $Y_2 = K_{\infty}'(t)$ otherwise.  If $t + 2(1-\eps)/n \leq t_1$ so that (\ref{bblexp}) can be applied and $n$ is large enough that (\ref{bblprob}) holds, then using that $\min\{\frac{a}{b},\frac{c}{d}\}\leq\frac{a+c}{b+d}$ for fractions of positive numbers to get the third inequality, we have
\begin{align*}
&E \bigg[ \bigg(K_n(t) - \frac{2}{t + 2/n} \bigg)^+ \bigg] \\
&\hspace{.5in}\leq E \bigg[ \bigg| K_{\infty} \Big( t + \frac{2(1 - \eps)}{n} \Big) - \frac{2}{t + 2/n} \bigg| \bigg] \\
&\hspace{1in}+ \big( n \1_{\{n \leq (2 + \eps)/t\}} + E[K_{\infty}'(t)]\1_{\{n > (2 + \eps)/t\}} \big) P \bigg(K_{\infty} \Big(\frac{2(1 - \eps)}{n} \Big) > n \bigg) \\
&\hspace{.5in}\leq \frac{\eps}{t + 2(1-\eps)/n} + \bigg( \frac{2}{t + 2(1 - \eps)/n} - \frac{2}{t + 2/n} \bigg) + \min\bigg\{n, \frac{2 + \eps}{t} \bigg\} \eps \\
&\hspace{.5in}\leq \frac{\eps}{(1 - \eps)(t + 2/n)} + \frac{2 \eps}{(1 - \eps)(t + 2/n)} + \frac{2(2 + \eps)\eps}{t + 2/n} \\
&\hspace{.5in}\leq \bigg(\frac{3 \eps}{1 - \eps} + 2(2 + \eps)\eps \bigg) \frac{1}{t + 2/n}.
\end{align*}
Combining these results gives that if $t + 2(1 + \eps)/n \leq t_1$ and $n$ is sufficiently large, we have $$E \bigg[ \bigg| K_n(t) - \frac{2}{t + 2/n} \bigg| \bigg] \leq \bigg(5 \eps + \frac{3 \eps}{1 - \eps} + 2(2 + \eps)\eps \bigg) \frac{1}{t + 2/n}.$$  The result follows.
\end{proof}

\begin{Cor}\label{kingcor}
Let $\delta > 0$, and choose $t_0$ and $M$ as in Lemma \ref{kinglem}.  Then for all $\eps \in (0,1)$, $t \leq t_0$, and $n \geq M$ such that $\eps > 2/(nt)$, we have
$$P \bigg( K_n(t) < \frac{2(1 - \eps)}{t} \bigg) \leq \frac{\delta}{\eps - 2/(nt)}.$$
\end{Cor}

\begin{proof}
By Lemma \ref{kinglem} and Markov's Inequality,
\begin{align*}
P\bigg( K_n(t) < \frac{2(1 - \eps)}{t} \bigg) &= P \bigg( K_n(t) < \frac{2}{t + 2/n} \cdot (1 - \eps) \Big(1 + \frac{2}{nt} \Big) \bigg) \\
&\leq P \bigg( \bigg| K_n(t) - \frac{2}{t + 2/n} \bigg| > \frac{2(1 - (1 - \eps) (1 + 2/(nt)))}{t + 2/n} \bigg) \\
&\leq \frac{\delta}{2(1 - (1 - \eps)(1 + 2/(nt)))} \\
&\leq \frac{\delta}{\eps - 2/(nt)},
\end{align*}
as claimed.
\end{proof}

\subsection{Proof of Proposition \ref{newprop}}\label{newpropproof}
\begin{proof}[Proof of Proposition \ref{newprop}]
We obtain upper and lower bounds on $N(t_j)$ by comparing our process to simpler coalescent processes.  For the upper bound, let $N^+(t)$ denote the number of individuals remaining at time $t$ in a model that is the same as our model, except that all species mergers are suppressed.  Suppressing species mergers can only reduce the number of mergers of individual lineages, so $N^+(t)$ stochastically dominates $N(t)$ for all $t$.  Let $N_i^+(t)$ be the number of individual lineages at time $t$ belonging to species $i$, in this new model.  Let $\varepsilon > 0$.  Then, using Markov's Inequality,
\begin{equation}\label{neweq1}
P \bigg( N^+(t_j) > \frac{(2 + \varepsilon)s_j}{t_j} \bigg) \leq \frac{t_j}{\varepsilon s_j} E \bigg[ \bigg| N^+(t_j) - \frac{2s_j}{t_j} \bigg| \bigg] \leq \frac{t_j}{\eps s_j} \sum_{i=1}^{s_j} E \bigg[ \bigg| N_i^+(t_j) - \frac{2}{t_j} \bigg| \bigg].
\end{equation}
Using Lemma \ref{kinglem} and the assumption that $1/n_j \ll t_j$, we get that for all $i \in \{1, \dots, s_j\}$ and all $\delta > 0$, 
\begin{equation}\label{neweq2}
\limsup_{j \rightarrow \infty} \: t_j E \bigg[ \bigg| N_i^+(t_j) - \frac{2}{t_j} \bigg| \bigg] \leq \limsup_{j \rightarrow \infty} \: t_j \bigg( \bigg| \frac{2}{t_j} - \frac{2}{t_j + 2/n_j} \bigg| + \frac{\delta}{t_j + 2/n_j} \bigg) = \delta.
\end{equation}
Combining (\ref{neweq1}) and (\ref{neweq2}) yields
\begin{equation}\label{newupper}
\lim_{j \rightarrow \infty} P \bigg( N^+(t_j) > \frac{(2 + \varepsilon)s_j}{t_j} \bigg) = 0.
\end{equation}

For the lower bound, recall that at time zero, blocks of the partition are labeled by the integers $1, \dots, s_j$, corresponding to the $s_j$ species.  When the two species corresponding to the labels $i$ and $j$ merge, where $i < j$, individuals of both species take the label $i$.  Let $N^-(t_j)$ denote the number of individual lineages at time $t_j$ whose species label has not changed between times $0$ and $t_j$.  That is, we keep only the individuals from one of the original species corresponding to each of the $S(t_j)$ species at time $t_j$.  Clearly $N^-(t_j) \leq N(t_j)$.  Conditional on $S(t_j) = s$, the distribution of $N^-(t_j)$ is the same as the distribution of what we get by running $s$ independent copies of Kingman's coalescent, each started with $n_j$ lineages, and counting the total number of lineages remaining at time $t_j$.  Therefore, the same reasoning that leads to (\ref{newupper}) gives
\begin{equation}\label{neweq3}
\lim_{j \rightarrow \infty} P \bigg( N^-(t_j) < \frac{(2 - \varepsilon)s}{t_j} \bigg| S(t_j) = s\bigg) = 0,
\end{equation}
and the convergence is uniform in $s$.  However, because $t_j \ll 1/s_j$, another application of Lemma~\ref{kinglem} yields $$\lim_{j \rightarrow \infty} P(S(t_j) \leq (1 - \eps) s_j) = 0.$$  Combining this result with (\ref{neweq3}) yields
\begin{equation}\label{newlower}
\lim_{j \rightarrow \infty} P \bigg( N^-(t_j) < \frac{(2 - \varepsilon)(1 - \varepsilon)s_j}{t_j} \bigg) = 0.
\end{equation}
The proposition follows from (\ref{newupper}) and (\ref{newlower}).
\end{proof}

\subsection{Kingman's coalescent and time-changed Yule trees}\label{KingmanYule}

We now define the coalescent process that describes the species mergers.  Let $(\Psi_{\infty}(t), t \geq 0)$ be a coalescent process having the same law as $(\Pi_{\infty}(ct), t \geq 0)$.  That is, $(\Psi_{\infty}(t), t \geq 0)$ has the same law as Kingman's coalescent, except that pairs of blocks merge at rate $c$ rather than at rate $1$.  For $s \in \N$, let $(\Psi_s(t), t \geq 0)$ denote the restriction of $(\Psi_{\infty}(t), t \geq 0)$ to $\{1, \dots, s\}$.  Let $S_{\infty}(t)$ be the number of blocks in the partition $\Psi_{\infty}(t)$, and let $S_s(t)$ denote the number of blocks in the partition $\Psi_s(t)$.  We interpret $S_{\infty}(t)$ as the number of species remaining at time $t$ when we start with infinitely many species at time zero, and $S_s(t)$ as the number of species remaining at time $t$ when we start with $s$ species at time zero.  Note that the coalescent process $(\Psi_{\infty}(t), t \geq 0)$ can also be depicted as a tree ${\cal T}$ with infinitely many leaves at height zero and $S_{\infty}(t)$ branches at height $t$.  The leaves can be labeled by the positive integers.

For positive integers $m$, let $\tau_m = \inf\{t: S_{\infty}(t) = m\}$.  If we consider the portion of the tree ${\cal T}$ below height $\tau_{m-1}$, we have $m$ subtrees, which we place in random order and denote by ${\cal T}^{1,m}, \dots, {\cal T}^{m,m}$.  One of these trees is pictured in Figure 4 below.

\begin{center}
\setlength{\unitlength}{0.9cm}
\begin{picture}(13,8.5)
\linethickness{0.3mm}
\put(1,1){\line(1,1){5.5}}
\put(6.5,6.5){\line(1,-1){5.5}}
\put(6.5,6.5){\line(0,1){1.5}}
\put(4,4){\line(1,-1){3}}
\put(10,3){\line(-1,-1){2}}
\put(2.5,2.5){\line(1,-1){1.5}}
\put(1.7,1.7){\line(1,-1){0.7}}
\put(6,2){\line(-1,-1){1}}
\put(11,2){\line(-1,-1){1}}
\put(8.7, 1.7){\line(1,-1){0.7}}
\put(3.6,1.4){\line(-1,-1){0.4}}
\put(5.5,1.5){\line(1,-1){0.5}}
\put(11.6,1.4){\line(-1,-1){0.4}}
\put(6.5,6.5){\circle*{0.1}}
\put(4,4){\circle*{0.1}}
\put(10,3){\circle*{0.1}}
\put(2.5,2.5){\circle*{0.1}}
\put(1.7,1.7){\circle*{0.1}}
\put(6,2){\circle*{0.1}}
\put(11,2){\circle*{0.1}}
\put(8.7,1.7){\circle*{0.1}}
\put(11.6,1.4){\circle*{0.1}}
\put(3.6,1.4){\circle*{0.1}}
\put(5.5,1.5){\circle*{0.1}}
\put(6.5,8){\circle*{0.1}}
\put(6.7,8){$\tau_{m-1}$}
\put(6.7,6.5){$V^{k,m}$}
\put(3,4){$V_1^{k,m}$}
\put(10.1,3){$V_2^{k,m}$}
\put(1.5,2.5){$V_{11}^{k,m}$}
\put(6.1,2){$V_{12}^{k,m}$}
\put(11.1,2){$V_{22}^{k,m}$}
\put(7.7,1.7){$V_{21}^{k,m}$}
\put(3.7,1.4){$V_{112}^{k,m}$}
\put(0.7,1.7){$V_{111}^{k,m}$}
\end{picture}

\vspace{-.2in}
Figure 4: The tree ${\cal T}^{k,m}$.
\end{center}

For $k \in \{1, \dots, m\}$, $\ell \in \{0, 1, 2, \dots\}$, and $i_1, \dots, i_\ell \in \{1, 2\}$, we will define random variables $U^{k,m}_{i_1 \dots i_{\ell}}$ and $V^{k,m}_{i_1 \dots i_{\ell}}$ as follows.
We begin at time $\tau_{m-1}$ and follow the tree ${\cal T}^{k,m}$ in reversed time from time $\tau_{m-1}$ down to time $0$, so that branches split instead of coalescing.  Define $V^{k,m}$ to be the time when the initial branch splits into two.  Then define $V^{k,m}_1$ and $V^{k,m}_2$ to be the times when the two branches created at time $V^{k,m}$, ordered at random, split again.  Given $V^{k,m}_{i_1 \dots i_l}$, let $V^{k,m}_{i_1 \dots i_l1}$ and $V^{k,m}_{i_1 \dots i_l 2}$ denote the times when the two branches created at time $V^{k,m}_{i_1 \dots i_l}$ split into two.  Let $U^{k,m} = V^{k,m}/\tau_{m-1}$, and for $\ell \geq 1$, define $U^{k,m}_{i_1 \dots i_{\ell}} = V^{k,m}_{i_1 \dots i_{\ell}}/V^{k,m}_{i_1 \dots i_{\ell-1}}$.  Then
\begin{equation}\label{Vdef}
V^{k,m}_{i_1 \dots i_{\ell}} = U^{k,m} U^{k,m}_{i_1} U^{k,m}_{i_1 i_2} \dots U^{k,m}_{i_1 \dots i_{\ell}} \tau_{m-1}.
\end{equation}

A key ingredient in our proof is that the random variables $U^{k,m}_{i_1 \dots i_{\ell}}$ are approximately independent, and have approximately a uniform distribution on $[0, 1]$.  Making this statement rigorous involves coupling the coalescent with a time-changed Yule process.  This connection between Kingman's coalescent and a Yule process was discussed in \cite{berestycki09}, in which both Kingman's coalescent and a Yule process are shown to be embedded in a Brownian excursion.

Consider a Yule process $(Y(t), t \geq 0)$, which is a continuous-time branching process in which there are no deaths and each individual independently gives birth at rate $1$.  Consider the time-change which maps $t$ to $u = 1 - e^{-t}$, so that $t = -\log(1 - u)$.  It is well-known that for all $u \in [0, 1)$, the next time that an individual at time $u$ gives birth is uniformly distributed on $[u, 1]$.  To see this, note that the probability that an individual at time $u$ gives birth before time $u + x(1 - u)$ is the same as the probability that an individual in the original Yule process at time $-\log(1 - u)$ gives birth before time $-\log(1 - u - x(1-u))$, which is $$1 - e^{\log(1 - u - x(1-u)) - \log(1 - u)} = 1 - \frac{1 - u - x(1-u)}{1 - u} = x.$$  We can then do the time-reversal $v = 1 - u = e^{-t}$, so $t = -\log v$.  After this additional time change, we start at time $1$, and individuals branch as we go backwards in time.  An individual at time $v$ will branch next at a time which is uniformly distributed on $[0, v]$, and individuals reproduce independently.

Fix a positive integer $m$.  We now obtain a Yule process started with $m$ individuals by starting with Kingman's coalescent and then performing a random time change.  

\begin{Lemma}\label{Yulelem}
For $0 < t \leq \tau_{m-1}$, let $$f_m(t) = \log \bigg(\frac{mc}{2} \bigg) + \int_t^{\tau_{m-1}} \frac{c (S_{\infty}(r) - 1)}{2} \: dr.$$  Then $(S_{\infty}(f_m^{-1}(u)), u > \log (mc/2))$ is a Yule process started with $m$ individuals at time $\log (mc/2)$.
\end{Lemma}

\begin{proof}
First, note that $f_m(t)$ is a strictly decreasing function of $t$, so the inverse function is well-defined. 
Now let $k \geq m$, and note that $\tau_{k-1} - \tau_k$, which is the amount of time for which there are $k$ species, has an exponential distribution with rate $ck(k-1)/2$.  Because the time change stretches time by a factor of $c(k-1)/2$ during this interval, the distribution of $f_m^{-1}(\tau_k) - f_m^{-1}(\tau_{k-1})$ is exponential with rate $k$, matching the distribution of the amount of time for which there are $k$ individuals in a Yule process.
\end{proof}

We can now make the further time change discussed in the paragraph before Lemma \ref{Yulelem}, and define $S^*(v) = S_{\infty}(f_m^{-1}(-\log v))$ for $0 < v < 2/cm$.  Just as there is a coalescent tree ${\cal T}$, with subtrees ${\cal T}^{1,m}, \dots, {\cal T}^{m,m}$, associated with the original coalescent process $(S_{\infty}(t), t \geq 0)$, there are $m$ subtrees ${\cal T}^{1,m,*}, \dots, {\cal T}^{m,m,*}$ associated with the process $(S^*(v), 0 < v < 2/cm)$, and we can use these trees to define associated random variables $U^{k,m,*}_{i_1 \dots i_l}$ and $V^{k,m,*}_{i_1 \dots i_l}$ as before.  Furthermore, because $(S^*(v), 0 < v < 2/cm)$ arises by time-changing a Yule process, it follows from the discussion above that the new random variables $U^{k,m,*}_{i_1, \dots, i_l}$ are independent, and each has exactly the uniform distribution on $[0, 1]$.

For $0 < t < \tau_{m-1}$, we have $S_{\infty}(t) = S^*(e^{-f_m(t)})$.  Lemmas \ref{timechange} and \ref{timechange2} below establish that this time change is only a small perturbation of time.

\begin{Lemma}\label{timechange}
We have $$\sup_{0 < t \leq \tau_{m-1}} \bigg| \frac{t}{e^{-f_m(t)}} - 1 \bigg| \rightarrow_{\mathrm{p}} 0,$$ where $\rightarrow_{\mathrm{p}}$ denotes convergence in probability as $m \rightarrow \infty$.
\end{Lemma}

\begin{proof}
Taking logarithms, it suffices to show that as $m \rightarrow \infty$,
\begin{equation}\label{fmabs}
\sup_{0 < t \leq \tau_{m-1}} |\log t + f_m(t)| \rightarrow_{\mathrm{p}} 0.
\end{equation}
From (\ref{cdi}), we have $t S_{\infty}(t) \rightarrow 2/c$ almost surely as $t \rightarrow 0$.  It follows that $\tau_k \sim 2/ck$ almost surely, where $\sim$ means that the ratio of the two sides tends to one as $k \rightarrow \infty$.  Taking logarithms,
\begin{equation}\label{taukest}
\lim_{k \rightarrow \infty} \bigg| \log \tau_k - \log \bigg( \frac{2}{ck} \bigg) \bigg| = 0 \hspace{.2in} \textup{a.s.}
\end{equation}
For $k \geq m-1$, let $$H_k = \int_{\tau_k}^{\tau_{m-1}} \frac{c(S_{\infty}(r) - 1)}{2} \: dr = \frac{c}{2} \sum_{j=m}^k (\tau_{j-1} - \tau_j) (j-1).$$  Because $\tau_{j-1} - \tau_j$ has an exponential distribution with rate parameter $cj(j-1)/2$, and these random variables are independent for different values of $j$, we have $$E[H_k] = \sum_{j=m}^k \frac{1}{j}$$ and $$\Var(H_k) = \sum_{j=m}^k \frac{1}{j^2}.$$  By Kolmogorov's Maximal Inequality applied to the independent mean zero random variables $H_k - H_{k-1} - 1/k$, we have for all $\eps > 0$,
\begin{equation}\label{Hkprob}
P \bigg( \sup_{k \geq m} \bigg|H_k - \sum_{j=m}^k \frac{1}{j} \bigg| > \eps \bigg) \leq \frac{1}{\eps^2} \sum_{j=m}^{\infty} \frac{1}{j^2} \leq \frac{1}{(m-1) \eps^2}.
\end{equation}
Now suppose $\tau_k \leq t < \tau_{k-1}$ for some $k > m$.  Then
\begin{align}\label{logtupper}
\log t + f_m(t) &\leq \log \tau_{k-1} + \log \bigg( \frac{mc}{2} \bigg) + H_k \nonumber \\
&\leq \log \bigg( \frac{2}{c(k-1)} \bigg) + \bigg| \log \tau_{k-1} - \log \bigg( \frac{2}{c(k-1)} \bigg) \bigg| \nonumber \\
&\hspace{1in}+ \log \bigg( \frac{mc}{2} \bigg) + \sum_{j=m}^k \frac{1}{j} + \bigg| H_k - \sum_{j=m}^k \frac{1}{j} \bigg| \nonumber \\
&= \sum_{j=m}^k \frac{1}{j} - \log \bigg( \frac{k-1}{m} \bigg) + \bigg| \log \tau_{k-1} - \log \bigg( \frac{2}{c(k-1)} \bigg) \bigg| + \bigg| H_k - \sum_{j=m}^k \frac{1}{j} \bigg|,
\end{align}
and likewise
\begin{align}\label{logtlower}
\log t + f_m(t) &\geq \log \tau_k + \log \bigg( \frac{mc}{2} \bigg) + H_{k-1} \nonumber \\
&\geq \sum_{j=m}^{k-1} \frac{1}{j} - \log \bigg( \frac{k}{m} \bigg) - \bigg| \log \tau_k - \log \bigg( \frac{2}{ck} \bigg) \bigg| - \bigg| H_{k-1} - \sum_{j=m}^{k-1} \frac{1}{j} \bigg|.
\end{align}
From (\ref{logtupper}) and (\ref{logtlower}), combined with the bounds (\ref{taukest}) and (\ref{Hkprob}) and standard estimates for the harmonic series, we obtain (\ref{fmabs}).
\end{proof}

\begin{Lemma}\label{timechange2}
We have $$\sup_{0 < u < v \leq \tau_{m-1}} \bigg| \frac{e^{-f_m(u)} - e^{-f_m(v)}}{u - v} - 1 \bigg| \rightarrow_{\mathrm{p}} 0,$$
where $\rightarrow_{\mathrm{p}}$ denotes convergence in probability as $m \rightarrow \infty$.
\end{Lemma}

\begin{proof}
Let $g(t) = e^{-f_m(t)}$.
We have $$\sup_{0 < u < v \leq \tau_{m-1}} \bigg| \frac{e^{-f_m(u)} - e^{-f_m(v)}}{u - v} - 1 \bigg| \leq \sup_{0 < t < \tau_m} |g'(t) - 1|.$$
Also, $$g'(t) = - e^{-f_m(t)} f_m'(t) = \frac{e^{-f_m(t)}}{t} \cdot \frac{c t (S_{\infty}(t) - 1)}{2}.$$ Equation (\ref{cdi}) implies that $ct(S_{\infty}(t) - 1)/2 \rightarrow 1$ almost surely as $t \rightarrow 0$, and also that $\tau_m \rightarrow 0$ almost surely as $m \rightarrow \infty$.  Combining these results with Lemma \ref{timechange}, we see that as $m \rightarrow \infty$, $$\sup_{0 < t < \tau_m} |g'(t) - 1| \rightarrow_{\mathrm{p}} 0,$$ which implies the result.
\end{proof}

\section{Results on the Nested Coalescent}\label{sec:nested}
\subsection{Convergence to a unique solution of the RDE}

Let ${\cal P}$ denote the set of probability distributions on $[2, \infty]$, and let ${\cal P}_1$ denote the set of probability distributions on $[2, \infty]$ with finite mean.  Let $T: {\cal P} \rightarrow {\cal P}$ be the mapping defined such that $T(\mu)$ is the distribution of
\begin{equation}\label{Tmu}
\frac{2}{1 - U(1 - \frac{2}{W_1 + W_2})},
\end{equation}
where $U$ has a uniform distribution on $[0,1]$, the random variables $W_1$ and $W_2$ have distribution $\mu$, and the random variables $U$, $W_1$, and $W_2$ are independent.  Let $T^n: {\cal P} \rightarrow {\cal P}$ be the map obtained by iterating $n$ times the map $T$.  Our goal in this subsection is to prove the following result.

\begin{Prop}\label{fixedpt}
The equation $T(\mu) = \mu$ has a unique solution $\mu^*$, and $\mu^* \in {\cal P}_1$.  For all $\mu \in {\cal P}$, the sequence $T^n(\mu)$ converges to $\mu^*$ in the sense of weak convergence of probability measures on $[2, \infty]$.  Also, the mean of $T^n(\mu)$ converges as $n \rightarrow \infty$ to the mean of $\mu^*$.
\end{Prop}

For $u \in (0, 1)$ and $x \in (0, \infty]$, define
\begin{equation}\label{hdef}
h(u,x) = \frac{2}{1 - u(1 - \frac{2}{x})} = \frac{2}{(1 - u) + \frac{2u}{x}}.
\end{equation}

\begin{Lemma}\label{finitemean}
We have $T({\cal P}_1) \subset {\cal P}_1$, and $T^2({\cal P}) \subset {\cal P}_1$.
\end{Lemma}

\begin{proof}
Let $\mu \in {\cal P}_1$.  Let $U$, $W_1$, and $W_2$ be independent random variables such that $U$ has a uniform distribution on $[0,1]$, and $W_1$ and $W_2$ have distribution $\mu$.  Then $T(\mu)$ has the same distribution as $h(U, W_1 + W_2)$, and a stochastic upper bound can be obtained by removing one of the two terms from the denominator on the right-hand side of (\ref{hdef}).  Therefore,
$$E[h(U, W_1 + W_2)] \leq E \bigg[ \frac{2}{1 - U} \1_{\{U \leq 1/2\}} \bigg] + E \bigg[ \frac{W_1 + W_2}{U} \1_{\{U \geq 1/2\}} \bigg] \leq 4 + 2 E[W_1 + W_2].$$
It follows that $T(\mu) \in {\cal P}_1$, which proves the first statement of the lemma.

Let $\delta_a$ denote the unit mass at $a$.  Because the expression in (\ref{Tmu}) is an increasing function of $W_1$ and $W_2$, if we can show that $T^2(\delta_{\infty}) \in {\cal P}_1$, then it will follow that $T^2(\mu) \in {\cal P}_1$ for all $\mu \in {\cal P}$, which will establish the second part of the lemma.  Note that $T(\delta_{\infty})$ has the same distribution as $2/(1 - U)$, which has the same distribution as $2/U$.  Therefore, $T^2(\delta_{\infty})$ has the same distribution as 
$$Y = \frac{2}{1 - U\big(1 - \frac{2}{2/U_1 + 2/U_2}\big)} = \frac{2}{(1 - U) + \frac{U}{1/U_1 + 1/U_2}},$$ where $U$, $U_1$, and $U_2$ are independent random variables, each having the uniform distribution on $[0,1]$.  Thus, it suffices to show that $E[Y] < \infty$.  We have $$Y \leq \min\bigg\{ \frac{2}{1 - U}, \: \frac{2}{U} \bigg( \frac{1}{U_1} + \frac{1}{U_2} \bigg) \bigg\}.$$  Let $x \geq 4$.  If $Y \geq x$, then we must have $2/(1-U) \geq x$ and therefore $U > 1 - 2/x$.  We also must have $(2/U)(1/U_1 + 1/U_2) \geq x$.  When $U > 1 - 2/x \geq 1/2$, this can only happen if $1/U_1 + 1/U_2 \geq x/4$, which requires either $U_1 \leq 8/x$ or $U_2 \leq 8/x$.  Thus,
$$P(Y \geq x) \leq P(U_1 \geq 1 - 2/x \mbox{ and either }U_1 \leq 8/x \mbox{ or }U_2 \leq 8/x) \leq \frac{2}{x} \bigg(\frac{8}{x} + \frac{8}{x} \bigg) = \frac{32}{x^2}.$$  It follows that 
$$E[Y] = \int_0^{\infty} P(Y \geq x) \: dx \leq 4 + \int_4^{\infty} \frac{32}{x^2} \: dx < \infty,$$
which completes the proof.
\end{proof}

Let $d$ denote the Kantorovich-Rubinstein metric on ${\cal P}_1$, which goes back to \cite{kantorovich58} and is also the Wasserstein metric for $p=1$.  That is, $$d(\mu, \nu) = \inf\{E[|X - Y|]: X \mbox{ has distribution } \mu \mbox{ and }Y \mbox{ has distribution } \nu\}.$$  It is well-known that $d$ is a complete metric on ${\cal P}_1$ (see, for example, \cite{givens84}).  Because $2 (1 - \log 2) < 1$, the following lemma shows that, with respect to this metric, $T$ is a strict contraction.

\begin{Lemma}\label{contract}
Suppose $\mu, \nu \in {\cal P}_1$ and $\mu \neq \nu$.  Then $$d(T(\mu), T(\nu)) \leq 2(1 - \log 2) d(\mu, \nu).$$
\end{Lemma}

\begin{proof}
Let $\eps > 0$.  Let $\mu, \nu \in {\cal P}_1$.  By the definition of $d$, on some probability space one can construct random variables $X_1$ and $Y_1$ such that $X_1$ has distribution $\mu$, $Y_1$ has distribution $\nu$, and $E[|X_1 - Y_1|] \leq d(\mu, \nu) + \eps$.  One can construct $X_2$ and $Y_2$, independently of $(X_1, Y_1)$, so that they satisfy these same conditions.  Let $U$ be a random variable that has a uniform distribution on $(0,1)$ and is independent of $(X_1, X_2, Y_1, Y_2)$.  Let $X = h(U, X_1 + X_2)$ and $Y = h(U, Y_1 + Y_2)$, where $h$ is the function defined in (\ref{hdef}).  Note that $X$ has the same distribution as $T(\mu)$, and $Y$ has the same distribution as $T(\nu)$.  For $x \geq 4$, $$\frac{\partial h}{\partial x}(u,x) = \frac{4u}{(x(1 - u) +2u)^2} \leq \frac{4u}{(4(1-u) + 2u)^2} = \frac{u}{4(1 - u/2)^2}.$$  Therefore, $$|X - Y| \leq \frac{U}{4(1 - U/2)^2} |(X_1 + X_2) - (Y_1 + Y_2)|.$$  Taking expectations, we get $$E[|X - Y|] \leq \frac{d(\mu, \nu) + \eps}{2} E \bigg[ \frac{U}{(1 - U/2)^2} \bigg] = \frac{d(\mu, \nu) + \eps}{2} \cdot 4(1 - \log 2).$$
Letting $\eps \rightarrow 0$ gives $E[|X - Y|] \leq 2(1 - \log 2) d(\mu, \nu)$, which implies the result.
\end{proof}

\begin{proof}[Proof of Proposition \ref{fixedpt}]
Because $T^2(\mu) \in {\cal P}_1$ for all $\mu \in {\cal P}$ (by Lemma \ref{finitemean}), any solution to the equation $T(\mu) = \mu$ must be in ${\cal P}_1$.  By Lemma \ref{contract}, the map $T$ is a strict contraction with respect to the Kantorovich-Rubinstein metric on ${\cal P}_1$.  Therefore, as noted in Lemma 5 of \cite{aldous05}, it follows from the Banach contraction theorem that the equation $T(\mu) = \mu$ has a unique solution $\mu^*$, and $T^n(\mu)$ converges to $\mu^*$ as $n \rightarrow \infty$ with respect to the Kantorovich-Rubinstein metric for all $\mu \in {\cal P}_1$.  Because convergence with respect to the Kantorovich-Rubinstein metric implies both weak convergence and convergence of means (see, for example, \cite{givens84}), the result follows.
\end{proof}

\subsection{Mergers of individual ancestral lines}

We now consider the merging of individual ancestral lines within a species.  Recall that, at time zero, there are $s_j$ species, and we sample $n_j$ individuals from each of the $s_j$ species.  Pairs of ancestral lines belonging to the same species merge at rate one.

Recall the definition of the trees ${\cal T}^{1,m}, \dots, {\cal T}^{m,m}$ derived from the species tree in section 2.3.  Let $N^{k,m}$ be the number of individual lineages remaining at time $\tau_{m-1}-$ that belong to the species represented by the tree ${\cal T}^{k,m}$. Note that this number could be zero when $s_j$ is finite because ${\cal T}^{k,m}$ is derived from a species tree starting from infinitely many species, whereas we only sample $n_j$ lineages from $s_j$ of these species.  Let $N^{k,m}_{i_1 \dots i_{\ell}}$ be the number of individual lineages, belonging to the species created by the merger at time $V_{i_1 \dots i_{\ell}}^{k,m}$, that remain at time $V^{k,m}_{i_1, \dots, i_{\ell-1}}-$.  If we know the values of $N^{k,m}_{i_1 \dots i_{\ell}}$ for all $i_1, \dots, i_{\ell} \in \{1,2\}$, then we obtain $N^{k,m}_{i_1 \dots i_{{\ell} -1}}$ by starting with $N^{k,m}_{i_1 \dots i_{\ell-1} 1} + N^{k,m}_{i_1 \dots i_{\ell-1} 2}$ lineages and running Kingman's coalescent for time $V^{k,m}_{i_1 \dots i_{\ell-2}} - V^{k,m}_{i_1 \dots i_{\ell-1}}$.  Also, let $W^{k,m}_{i_1 \dots i_\ell} = V^{k,m}_{i_1 \dots i_{\ell-1}} N^{k,m}_{i_1 \dots i_{\ell}}$ for $\ell \geq 1$, and let $W^{k,m} = \tau_{m-1} N^{k,m}$.

Fix a positive integer $d$.  Let $W_{i_1 \dots i_d}^{k,m,*,+} = \infty$, and let $W_{i_1 \dots i_d}^{k,m,*,-} = 2$.  Let $W_{i_1 \dots i_d}^{k,m,*}$ and $\bar{W}_{i_1 \dots i_d}^{k,m}$ both equal 
$\max\{W_{i_1 \dots i_d}^{k,m}, 2\}$.  Recall the definition of the function $h$ from (\ref{hdef}).  For $0 \leq \ell \leq d-1$ and $i_1, \dots, i_{\ell} \in \{1, 2\}$, let
\begin{align*}
W_{i_1 \dots i_{\ell}}^{k,m,*} &= h(U_{i_1 \dots i_{\ell}}^{k,m,*}, W_{i_1 \dots i_{\ell} 1}^{k,m,*} + W_{i_1 \dots i_{\ell}2}^{k,m,*}), \\
W_{i_1 \dots i_{\ell}}^{k,m,*,+} &= h(U_{i_1 \dots i_{\ell}}^{k,m,*}, W_{i_1 \dots i_{\ell} 1}^{k,m,*,+} + W_{i_1 \dots i_{\ell}2}^{k,m,*,+}), \\
W_{i_1 \dots i_{\ell}}^{k,m,*,-} &= h(U_{i_1 \dots i_{\ell}}^{k,m,*}, W_{i_1 \dots i_{\ell} 1}^{k,m,*,-} + W_{i_1 \dots i_{\ell}2}^{k,m,*,-}), \\
{\bar W}_{i_1 \dots i_{\ell}}^{k,m} &= h(U_{i_1 \dots i_{\ell}}^{k,m}, {\bar W}_{i_1 \dots i_{\ell} 1}^{k,m} + {\bar W}_{i_1 \dots i_{\ell}2}^{k,m}).
\end{align*}
Because the random variables $U^{k,m,*}_{i_1 \dots i_{\ell}}$ are independent and have a uniform distribution on $[0,1]$, the distribution of $W_{i_1 \dots i_{d-1}}^{k,m,*,+}$ is $T(\delta_{\infty})$, while the distribution of $W_{i_1 \dots i_{d-1}}^{k,m,-}$ is $T(\delta_2)$.  More generally, for $0 \leq \ell \leq d$, the distributions of $W_{i_1 \dots i_{\ell}}^{k,m,*,+}$ and $W_{i_1 \dots i_{\ell}}^{k,m,*,-}$ are $T^{d - \ell}(\delta_{\infty})$ and $T^{d - \ell}(\delta_2)$ respectively.  In particular, the distributions of $W^{k,m,*,+}$ and $W^{k,m,*,-}$ are $T^d(\delta_{\infty})$ and $T^d(\delta_2)$ respectively.  Also, because $h(u,x)$ is an increasing function of $x$, we have
\begin{equation}\label{Wcompare}
W^{k,m,*,-} \leq W^{k,m,*} \leq W^{k,m,*,+}.
\end{equation} 

To prove Theorem \ref{Speed}, we will consider a sequence $(m_j)_{j=1}^{\infty}$ tending to infinity.  That is, for the process in which there are $s_j$ species and $n_j$ individuals sampled from each of these species, we will consider the trees ${\cal T}^{k, m_j}$.  Throughout the rest of this section, we will occasionally drop the superscripts $k$ and $m_j$ to lighten notation, when doing so seems unlikely to cause confusion.

\begin{Lemma}\label{starbar}
We have $$\sup_{1 \leq k \leq m_j} \bigg| \frac{{\bar W}^{k,m_j}}{W^{k,m_j,*}} - 1 \bigg| \rightarrow_{\mathrm{p}} 0,$$ where $\rightarrow_{\mathrm{p}}$ denotes convergence in probability as $j \rightarrow \infty$.
\end{Lemma}

\begin{proof}
Recall the definition of the function $h$ from (\ref{hdef}).  Note that
$$\frac{h(u_1, x_1)}{h(u_2, x_2)} = \frac{(1 - u_2) + 2u_2/x_2}{(1 - u_1) + 2u_1/x_1}.$$
Therefore, using that $\min\{\frac{a}{b},\frac{c}{d}\}\leq\frac{a+c}{b+d}\leq\max\{\frac{a}{b},\frac{c}{d}\}$,
\begin{align*}
\frac{h(u_1, x_1)}{h(u_2, x_2)} &
\leq\max\bigg\{ \frac{1 - u_2}{1 - u_1}, \frac{u_2x_1}{u_1x_2} \bigg\}\\
&\leq\max\bigg\{ \frac{1 - u_2}{1 - u_1}, \left(\frac{u_2}{u_1}\right)^2,\left(\frac{x_1}{x_2}\right)^2 \bigg\}\\
&\leq \max\bigg\{ \frac{1 - u_2}{1 - u_1}, \frac{u_2}{u_1}, \frac{x_1}{x_2} \bigg\}^2
\end{align*}
and
$$\frac{h(u_1, x_1)}{h(u_2, x_2)} \geq\min\bigg\{ \frac{1 - u_2}{1 - u_1}, \frac{u_2}{u_1}, \frac{x_1}{x_2} \bigg\}^2.$$
Recall that $W_{i_1 \dots i_d}^{*} = {\bar W}_{i_1 \dots i_d}$, and for $0 \leq \ell \leq d-1$, we have
\begin{equation}\label{hratio}
\frac{{\bar W}_{i_1 \dots i_{\ell}}}{W^*_{i_1 \dots i_{\ell}}} = \frac{h(U_{i_1 \dots i_{\ell}}, {\bar W}_{i_1 \dots i_{\ell} 1} + {\bar W}_{i_1 \dots i_{\ell} 2})}{h(U^*_{i_1 \dots i_{\ell}}, W^*_{i_1 \dots i_{\ell} 1} + W^*_{i_1 \dots i_{\ell} 2})}.
\end{equation}
Recall also that, defining the random function $f_m$ as in Lemma \ref{Yulelem} and defining $V^*_{i_1 \dots i_{\ell}}$ as in the discussion following that lemma, we have
$V_{i_1 \dots i_{\ell}}^* = e^{-f_{m_j}(V_{i_1 \dots i_{\ell}})}$.  For $1 \leq \ell \leq d-1$, we have
\begin{equation}\label{Uratio}
\frac{U^*_{i_1 \dots i_{\ell}}}{U_{i_1 \dots i_{\ell}}} = \frac{V^*_{i_1 \dots i_{\ell}}}{V^*_{i_1 \dots i_{\ell - 1}}} \cdot \frac{V_{i_1 \dots i_{\ell - 1}}}{V_{i_1 \dots i_{\ell}}} = \frac{e^{-f_{m_j}(V_{i_1 \dots i_{\ell}})}}{V_{i_1 \dots i_{\ell}}} \cdot \frac{V_{i_1 \dots i_{\ell - 1}}}{e^{-f_{m_j}(V_{i_1 \dots i_{\ell-1}})}}
\end{equation}
and
$$\frac{1 - U^*_{i_1 \dots i_{\ell}}}{1 - U_{i_1 \dots i_{\ell}}} = \frac{(V^*_{i_1 \dots i_{\ell - 1}} - V^*_{i_1 \dots i_{\ell}})V_{i_1 \dots i_{\ell-1}}}{V^*_{i_1 \dots i_{\ell - 1}}(V_{i_1 \dots i_{\ell - 1}} - V_{i_1 \dots i_{\ell}})} = \frac{e^{-f_{m_j}(V_{i_1 \dots i_{\ell-1}})} - e^{-f_{m_j}(V_{i_1 \dots i_{\ell}})}}{V_{i_1 \dots i_{\ell - 1}} - V_{i_1 \dots i_{\ell}}} \cdot \frac{V_{i_1 \dots i_{\ell - 1}}}{e^{-f_{m_j}(V_{i_1 \dots i_{\ell - 1}})}}.$$  Likewise, for the $\ell = 0$ case,
$$\frac{U^*}{U} = \frac{e^{-f_{m_j}(V)}}{V} \cdot \frac{\tau_{m_j-1}}{e^{-f_{m_j}(\tau_{m_j-1})}}$$ and
$$\frac{1 - U^*}{1 - U} = \frac{e^{-f_{m_j}(\tau_{m_j-1})} - e^{-f_{m_j}(V)}}{\tau_{m_j-1} - V} \cdot \frac{\tau_{m_j-1}}{e^{-f_{m_j}(\tau_{m_j-1})}}.$$  Let $$R = \max \bigg\{ \sup_{0 < t \leq \tau_{m_j-1}} \bigg| \frac{t}{e^{-f_{m_j}(t)}} - 1 \bigg|, \: \sup_{0 < u < v \leq \tau_{m_j-1}} \bigg|\frac{e^{-f_{m_j}(u)} - e^{-f_{m_j}(v)}}{u - v} - 1 \bigg| \bigg\},$$ which converges in probability to zero as $j \rightarrow \infty$ by Lemmas \ref{timechange} and \ref{timechange2}.  Using the fact that if $|a-1| \leq R$ then $a$ and $1/a$ are both between $1-R$ and $1/(1-R)$, it follows from
these results with (\ref{hratio}), we obtain $$(1-R)^4 \leq \frac{{\bar W}_{i_1 \dots i_{d-1}}}{W^*_{i_1 \dots i_{d-1}}} \leq \bigg( \frac{1}{1 - R} \bigg)^4.$$  Then by induction, we end up with $$(1-R)^{2^{d+1}} \leq \frac{{\bar W^{k,m_j}}}{W^{k,m_j,*}} \leq \bigg( \frac{1}{1 - R} \bigg)^{2^{d+1}}$$ for all $k \in \{1, \dots, m_j\}$.  Because $d$ is a fixed positive integer, the result follows. 
\end{proof}

\begin{Lemma}\label{Rlem}
Suppose $s_j \gg m_j$.  Let $\eps > 0$.
For $i_1, \dots, i_d \in \{1, 2\}$ and $k \in \{1, \dots, m_j\}$, let $R^{k,m_j}_{i_1 \dots i_d}$ be the number of species, among the $s_j$ present at time zero, that are descended from the species created by the merger at time $V^{k,m_j}_{i_1 \dots i_d}$.  Then there exists $\delta > 0$ such that for sufficiently large $j$, we have $$P\bigg( R^{k,m_j}_{i_1 \dots i_d} > \frac{\delta s_j}{m_j} \bigg) > 1 - \eps.$$
\end{Lemma}

\begin{proof}
{ For all $t \geq 0$, the partition given by Kingman's coalescent at time $t$, $\Pi_{\infty}(t)$, is an exchangeable random partition of $\N$.  Therefore, if $B$ is a block of the partition $\Pi_{\infty}(t)$, then the limit $$\lim_{n \rightarrow \infty} \frac{1}{n} \sum_{i=1}^n \1_{\{i \in B\}}$$ exists and is called the asymptotic frequency of $B$.  Let $K_{\infty}(t)$ be the number of blocks of $\Pi_{\infty}(t)$, and let $\tau_m = \inf\{t: K_{\infty}(t) = m\}$ be the first time that the coalescent has $m$ blocks.  Denote by $\Lambda(t)$ the sequence consisting of the asymptotic frequencies of the blocks of $\Pi_{\infty}(t)$, ranked in decreasing order.  It is shown in \cite{kingman82} that the distribution of $\Lambda(\tau_m)$ is uniform on the simplex $$\Delta_m = \{x_1 \geq \dots \geq x_m \geq 0: x_1 + \dots + x_m = 1\}.$$
In particular, if we choose one of the $m$ blocks uniformly at random, the distribution of the asymptotic frequency of this block is Beta$(1, m-1)$.  Furthermore, if we follow Kingman's coalescent in reversed time, so that blocks split instead of merging, and $B$ is a block with asymptotic frequency $\lambda$, then immediately after this block splits into two, the new blocks will have asymptotic frequencies $\lambda U$ and $\lambda(1-U)$, where $U$ has a uniform distribution on $[0,1]$.}

By the discussion above, the asymptotic frequency of the block of $\Psi_{\infty}(\tau_{m_j})$ corresponding to the species represented by the tree ${\cal T}^{k, m_j}$ has the Beta$(1, m_j - 1)$ distribution.  Moreover, let $\Lambda_{i_1 \dots i_d}$ be the asymptotic frequency of the block of $\Psi_{\infty}(V_{i_1 \dots i_d})$ created by the merger at time $V_{i_1 \dots i_d}$.  Then the distribution of $\Lambda_{i_1 \dots i_d}$ is the same the distribution of the product of $d+1$ independent random variables, one of them having the Beta$(1, m_j-1)$ distribution and $d$ of them having the Uniform$(0,1)$ distribution.  Because $d$ is a fixed positive integer, it follows that there exists $\delta > 0$ such that for all $j$, we have 
$$P\bigg(\Lambda_{i_1 \dots i_d} > \frac{2\delta}{m_j} \bigg) > 1 - \frac{\eps}{2}.$$ 
Conditional on $\Lambda_{i_1 \dots i_d}$, the distribution of $R_{i_1 \dots i_d}$ is Binomial$(s_j, \Lambda_{i_1 \dots i_d})$.  Because $s_j \gg m_j$, the result now follows from elementary concentration results for the binomial distribution.
\end{proof}

\begin{Lemma}\label{Llem}
Let
\begin{equation}\label{Lkmdef}
L^{k,m_j} = \max_{i_1, \dots, i_d \in \{1, 2\}} \bigg| \frac{W^{k,m_j}_{i_1 \dots i_d}}{{\bar W}^{k,m_j}_{i_1 \dots i_d}} - 1 \bigg|.
\end{equation} 
Suppose $s_j \gg m_j$ and $s_j n_j \gg m_j^2$.  Then $$\lim_{j \rightarrow \infty} E[L^{k,m_j} W^{k,m_j,*,+}] = 0.$$  
\end{Lemma}

\begin{proof}
Note that $W_{i_1 \dots i_d} = {\bar W}_{i_1 \dots i_d}$ unless $W_{i_1 \dots i_d} < 2$.  Therefore $|L^{k,m_j}| \leq 1$, and $L^{k,m_j} = 0$ unless $W_{i_1 \dots i_d} < 2$ for some $i_1, \dots, i_d$.  Because the distribution of $W^{k,m_j,*,+}$ is exactly $T^d(\delta_{\infty})$ for all $k$ and $m_j$, the collection of random variables $\{L^{k,m_j} W^{k,m_j,*,+}: j \in \N, 1 \leq k \leq m_j\}$ is uniformly integrable.  Therefore, noting also that the distribution of $L^{k,m_j} W^{k,m_j,*,+}$ does not depend on $k$, it suffices to show that $L^{k,m_j} W^{k,m_j,*,+} \rightarrow_{\mathrm{p}} 0$ as $m \rightarrow \infty$.  Because the random variables $W^{k,m_j,*,+}$ are identically distributed and finite, it suffices to show that $L^{k,m_j} \rightarrow_{\mathrm{p}} 0$ as $j \rightarrow \infty$.

Let $\eps > 0$.  We have
\begin{align}\label{Leq1}
P(L^{k,m_j} > \eps) &\leq 2^d P \bigg( \bigg| \frac{W_{i_1 \dots i_d}}{{\bar W}_{i_1 \dots i_d}} - 1 \bigg| > \eps \bigg) \nonumber \\
&= 2^d P\big(W_{i_1 \dots i_d} < 2(1 - \eps)\big) \nonumber \\
&= 2^d P\big(V_{i_1 \dots i_{d-1}} N_{i_1 \dots i_d} < 2(1 - \eps) \big).
\end{align}

Recall the definition of $R_{i_1 \dots i_d}$ from Lemma \ref{Rlem}.  Note that there are $R_{i_1 \dots i_d} n_j$ individual lineages at time zero descended from the species created by the merger at time $V_{i_1 \dots i_d}$.  Pairs of these individual lineages are subject to mergers at rate one, once the corresponding species lineages have merged, which means we can obtain a stochastic lower bound on the number of individual lineages by allowing all pairs of these lineages to merge at rate one.  Therefore, a stochastic lower bound for $N_{i_1 \dots i_d}$ can be obtained first constructing the species tree and then running the block-counting process associated with Kingman's coalescent, started with $R_{i_1 \dots i_d} n_j$ lineages, for time $V_{i_1 \dots i_{d-1}}$.  In particular, denoting by ${\cal G}$ the $\sigma$-field generated by 
the process $(\Psi_{\infty}(t), t \geq 0)$ that governs the species mergers, we have
$$P \big( V_{i_1 \dots i_{d-1}} N_{i_1 \dots i_d} < 2(1 - \eps) \big| {\cal G} \big) \leq P \big( V_{i_1 \dots i_{d-1}} K_{R_{i_1 \dots i_d} n_j}(V_{i_1 \dots i_{d-1}}) < 2(1 - \eps) \big| {\cal G} \big).$$
Now let $\delta = \eps^2$ and apply Corollary \ref{kingcor} with $V_{i_1 \dots i_{d-1}}$ in place of $t$ and $R_{i_1 \dots i_d} n_i$ in place of $n$ to get
\begin{equation}\label{Leq2}
P \big( V_{i_1 \dots i_{d-1}} K_{R_{i_1 \dots i_d} n_j}(V_{i_1 \dots i_{d-1}}) < 2(1 - \eps) \big| {\cal G} \big) \leq \frac{\eps^2}{\eps - 2/(V_{i_1 \dots i_{d-1}} R_{i_1 \dots i_d} n_i)}
\end{equation}
on the event that $V_{i_1 \dots i_{d-1}} < t_0$, $R_{i_1 \dots i_d} n_j \geq M$, and $2/(V_{i_1 \dots i_{d-1}} R_{i_1 \dots i_d} n_j) < \eps$.  Note that $P(V_{i_1 \dots i_{d-1}} < t_0) \geq P(\tau_{m_j-1} < t_0) \rightarrow 1$ as $j \rightarrow \infty$ by (\ref{cdi}).  Therefore, the result that $L^{k,m_j} \rightarrow_{\mathrm{p}} 0$, and therefore the result of the lemma, will follow from (\ref{Leq1}) and (\ref{Leq2}) provided we can show that
$$V_{i_1 \dots i_{d-1}} R_{i_1 \dots i_d} n_j \rightarrow_{\mathrm{p}} \infty, \hspace{.2in}\mbox{as }j \rightarrow \infty.$$

Recall from equation (\ref{Vdef}) that $V_{i_1 \dots i_{d-1}} = U U_{i_1} U_{i_1 i_2} \dots U_{i_1 \dots i_{d-1}} \tau_{m_j-1}$.  It follows from (\ref{cdi}) that $m_j \tau_{m_j-1} \rightarrow 2/c$ almost surely as $j \rightarrow \infty$.  Combining this observation with (\ref{Uratio}) and Lemma~\ref{timechange}, we see that there is a constant $\delta_1 > 0$ such that $P(V_{i_1 \dots i_{d-1}} > \delta_1/m_j) > 1-\eps/2$ for sufficiently large $j$.  By Lemma \ref{Rlem} and the assumption that $s_j \gg m_j$, there is a constant $\delta_2 > 0$ such that $P(R_{i_1 \dots i_d} > \delta_2 s_j/m_j) > 1 - \eps/2$ for sufficiently large $j$.  Combining these results, we get
$$P \bigg( V_{i_1 \dots i_{d-1}} R_{i_1 \dots i_d} n_j > \frac{\delta_1 \delta_2 s_j n_j}{m_j^2} \bigg) > 1 - \eps$$ for sufficiently large $j$.  Because $s_j n_j \gg m_j^2$ by assumption, the result follows. 
\end{proof}

\begin{Lemma}\label{WWstar}
Let $\eps > 0$.  Suppose $s_j \gg m_j$ and $s_j n_j \gg m_j^2$. There is a positive constant $t_0$, depending on $\eps$, such that if we define the events
$$A_j = \{\tau_{m_j-1} \leq t_0\}, \hspace{.5in}B_j = \bigg\{\sup_{1 \leq k \leq m_j} \bigg| \frac{{\bar W}^{k,m_j}}{W^{k,m_j,*}} - 1 \bigg| \leq \frac{1}{2} \bigg\},$$ then for sufficiently large $j$, we have
$$E \big[ |W^{k,m_j} - {\bar W}^{k,m_j}| \1_{A_j \cap B_j} \big] < \eps.$$
\end{Lemma}

\begin{proof}
By Proposition \ref{fixedpt}, we can choose a positive integer $d$ large enough that the mean of the distribution $T^d(\delta_{\infty})$ is less than $4 \gamma/3$.  Choose $0 < \delta < 1$ small enough that $(\delta/2) \sum_{n=0}^{d-1} (2 + \delta)^n < \eps/2$ and $6^d \gamma \delta < \eps/4$.  
Choose a positive integer $M$ and then choose $t_0 < \delta/M$ such that if $t \leq t_0$ and $n \geq M$, then the conclusion of Lemma \ref{kinglem} holds for this choice of $\delta$.

Suppose $0 \leq \ell \leq d - 1$.  Then, dropping the superscripts $k$ and $m_j$ to lighten notation,
$$\bigg| \frac{W_{i_1 \dots i_{\ell}}}{{\bar W}_{i_1 \dots i_{\ell}}} - 1 \bigg| = \bigg| \frac{W_{i_1 \dots i_{\ell}}}{h(U_{i_1 \dots i_{\ell}}, W_{i_1 \dots i_{\ell} 1} + W_{i_1 \dots i_{\ell} 2})} \cdot \frac{{h(U_{i_1 \dots i_{\ell}}, W_{i_1 \dots i_{\ell} 1} + W_{i_1 \dots i_{\ell} 2})}}{{h(U_{i_1 \dots i_{\ell}}, {\bar W}_{i_1 \dots i_{\ell} 1} + {\bar W}_{i_1 \dots i_{\ell} 2})}} - 1 \bigg|.$$
Recall the definition of the function $h$ from (\ref{hdef}).  Because $$|xy - 1| = |x - 1 + x(y-1)| \leq |x - 1| + x|y-1|$$ for positive real numbers $x$ and $y$, we have
\begin{align*}
\bigg| \frac{W_{i_1 \dots i_{\ell}}}{{\bar W}_{i_1 \dots i_{\ell}}} - 1 \bigg| &\leq \bigg| \frac{W_{i_1 \dots i_{\ell}}}{h(U_{i_1 \dots i_{\ell}}, W_{i_1 \dots i_{\ell} 1} + W_{i_1 \dots i_{\ell} 2})} - 1 \bigg| \\
&\hspace{.3in}+ \frac{W_{i_1 \dots i_{\ell}}}{h(U_{i_1 \dots i_{\ell}}, W_{i_1 \dots i_{\ell} 1} + W_{i_1 \dots i_{\ell} 2})} \bigg| \frac{{h(U_{i_1 \dots i_{\ell}}, W_{i_1 \dots i_{\ell} 1} + W_{i_1 \dots i_{\ell} 2})}}{{h(U_{i_1 \dots i_{\ell}}, {\bar W}_{i_1 \dots i_{\ell} 1} + {\bar W}_{i_1 \dots i_{\ell} 2})}} - 1 \bigg|.
\end{align*}
If $0 \leq u \leq 1$ and $x_1, x_2, y_1, y_2 > 0$, then
 $$\bigg| \frac{h(u, x_1 + x_2)}{h(u, y_1 + y_2)} - 1 \bigg| = \bigg| \frac{(1 - u) + \frac{2u}{y_1 + y_2}}{(1 - u) + \frac{2u}{x_1 + x_2}} - 1 \bigg| \leq \bigg| \frac{x_1 + x_2}{y_1 + y_2} - 1 \bigg| \leq \bigg| \frac{x_1}{y_1} - 1 \bigg| + \bigg| \frac{x_2}{y_2} - 1 \bigg|.$$  Therefore,
\begin{align}\label{Wlbound}
\bigg| \frac{W_{i_1 \dots i_{\ell}}}{{\bar W}_{i_1 \dots i_{\ell}}} - 1 \bigg| &\leq \bigg| \frac{W_{i_1 \dots i_{\ell}}}{h(U_{i_1 \dots i_{\ell}}, W_{i_1 \dots i_{\ell} 1} + W_{i_1 \dots i_{\ell} 2})} - 1 \bigg| \nonumber \\
&\hspace{.5in}+ \frac{W_{i_1 \dots i_{\ell}}}{h(U_{i_1 \dots i_{\ell}}, W_{i_1 \dots i_{\ell} 1} + W_{i_1 \dots i_{\ell} 2})} \bigg( \bigg| \frac{W_{i_1 \dots i_{\ell} 1}}{{\bar W}_{i_1 \dots i_{\ell} 1}} - 1 \bigg| + \bigg| \frac{W_{i_1 \dots i_{\ell} 2}}{{\bar W}_{i_1 \dots i_{\ell} 2}} - 1 \bigg| \bigg).
\end{align}
Interpreting $V_{i_1 \dots i_{\ell-1}}$ to be $\tau_{m_j-1}$ when $\ell = 0$, we have
\begin{align*}
\bigg| \frac{W_{i_1 \dots i_{\ell}}}{h(U_{i_1 \dots i_{\ell}}, W_{i_1 \dots i_{\ell} 1} + W_{i_1 \dots i_{\ell} 2})} - 1 \bigg| &= \bigg| \frac{N_{i_1 \dots i_{\ell}} V_{i_1 \dots i_{\ell - 1}}}{h(V_{i_1 \dots i_{\ell}}/V_{i_1 \dots i_{\ell - 1}}, V_{i_1 \dots i_{\ell}}(N_{i_1 \dots i_{\ell} 1} + N_{i_1 \dots i_{\ell} 2}))} - 1 \bigg| \\
&= \bigg| N_{i_1 \dots i_{\ell}} \bigg( \frac{(V_{i_1 \dots i_{\ell-1}} - V_{i_1 \dots i_{\ell}}) + 2/(N_{i_1 \dots i_{\ell} 1} + N_{i_1 \dots i_{\ell} 2})}{2} \bigg) - 1 \bigg|.
\end{align*}
Recall that $N_{i_1 \dots i_{\ell}}$ is obtained by running Kingman's coalescent started with $N_{i_1 \dots i_{\ell} 1} + N_{i_1 \dots i_{\ell} 2}$ blocks for time $V_{i_1 \dots i_{\ell-1}} - V_{i_1 \dots i_\ell}$. 
Now let ${\cal F}_{\ell,j}$ denote the $\sigma$-field generated by the process $(\Psi_{\infty}(t), t \geq 0)$
and the random variables $W_{i_1 \dots i_b}$ with $\ell+1 \leq b \leq d$ and $i_1 \dots i_b \in \{1, 2\}$.  By Lemma \ref{kinglem}, 
$$E \bigg[ \bigg| \frac{W_{i_1 \dots i_{\ell}}}{h(U_{i_1 \dots i_{\ell}}, W_{i_1 \dots i_{\ell} 1} + W_{i_1 \dots i_{\ell} 2})} - 1 \bigg| \1_{A_j} \Big| {\cal F}_{\ell, j} \bigg] \leq \frac{\delta}{2}$$ on the event $D_{i_1 \dots i_{\ell}} = \{N_{i_1 \dots i_{\ell} 1} + N_{i_1 \dots i_{\ell} 2} \geq M\}$.
Combining this result with (\ref{Wlbound}) yields that on $D_{i_1 \dots i_{\ell}}$,
\begin{equation}\label{newWrat}
E \bigg[ \bigg| \frac{W_{i_1 \dots i_{\ell}}}{{\bar W}_{i_1 \dots i_{\ell}}} - 1 \bigg| \1_{A_j} \Big| {\cal F}_{\ell, j} \bigg] \leq \frac{\delta}{2} + \bigg(1 + \frac{\delta}{2} \bigg) \bigg( \bigg| \frac{W_{i_1 \dots i_{\ell} 1}}{{\bar W}_{i_1 \dots i_{\ell} 1}} - 1 \bigg| + \bigg| \frac{W_{i_1 \dots i_{\ell} 2}}{{\bar W}_{i_1 \dots i_{\ell} 2}} - 1 \bigg| \bigg) \1_{A_j}.
\end{equation}
Now suppose $D_{i_1 \dots i_{\ell}}^c$ occurs.  Then $W_{i_1 \dots i_{\ell}} = V_{i_1 \dots i_{\ell}} N_{i_1 \dots i_{\ell}} \leq V_{i_1 \dots i_{\ell}}(N_{i_1 \dots i_{\ell 1}} + N_{i_1 \dots i_{\ell} 2}) \leq \tau_{m_j - 1} M$.
Because ${\bar W}_{i_1 \dots i_{\ell}} \geq 2$ by construction in view of the definition of the function $h$, it follows that
$W_{i_1 \dots i_{\ell}}/{\bar W}_{i_1 \dots i_{\ell}} \leq \tau_{m_j - 1}M/2 \leq M t_0/2 \leq \delta/2$ on $A_j$.  By the same reasoning, if $N_{i_1 \dots i_{\ell} 1} < M$, we have $W_{i_1 \dots i_{\ell} 1}/{\bar W}_{i_1 \dots i_{\ell} 1} \leq \delta/2$ on $A_j$, and likewise if $N_{i_1 \dots i_{\ell} 2} < M$.  Thus, on the event $D_{i_1 \dots i_{\ell}}^c$, the left-hand side of (\ref{newWrat}) is bounded above by 1, while the right-hand side is bounded below by $\delta/2 + (1 + \delta/2)(1 - \delta/2) \geq 1$ on $A_j$.  Therefore, (\ref{newWrat}) also holds on $D_{i_1, \dots i_{\ell}}^c$.
Now taking conditional expectations with respect to ${\cal F}_{d-1,j}$ on both sides of (\ref{newWrat}), we get
\begin{equation}\label{Winduct}
E \bigg[ \bigg| \frac{W_{i_1 \dots i_{\ell}}}{{\bar W}_{i_1 \dots i_{\ell}}} - 1 \bigg| \1_{A_j} \Big| {\cal F}_{d-1,j} \bigg] \leq \frac{\delta}{2} + (2 + \delta) E \bigg[ \bigg| \frac{W_{i_1 \dots i_{\ell} 1}}{{\bar W}_{i_1 \dots i_{\ell} 1}} - 1 \bigg| \1_{A_j} \Big| {\cal F}_{d-1,j} \bigg].
\end{equation}

Recall the definition of $L^{k,m_j}$ from (\ref{Lkmdef}).  Note that $L^{k,m_j}$ is ${\cal F}_{d-1,j}$-measurable and $A_j \in {\cal F}_{d-1,j}$.  Therefore, when $\ell = d-1$, equation (\ref{Winduct}) implies
$$E \bigg[ \bigg| \frac{W_{i_1 \dots i_{d-1}}}{{\bar W}_{i_1 \dots i_{d-1}}} - 1 \bigg| \1_{A_j} \Big| {\cal F}_{d-1,j} \bigg] \leq \bigg( \frac{\delta}{2} + (2 + \delta) L^{k,m_j}\bigg) \1_{A_j}.$$  Applying (\ref{Winduct}) inductively as $\ell$ goes from $d-2$ down to $0$ gives
$$E \bigg[ \bigg| \frac{W^{k,m_j}}{{\bar W}^{k,m_j}} - 1 \bigg| \1_{A_j} \Big| {\cal F}_{d-1,j} \bigg] \leq \bigg( \frac{\delta}{2} \sum_{n=0}^{d - 1} (2 + \delta)^n + (2 + \delta)^d L^{k,m_j} \bigg) \1_{A_j} \leq \bigg( \frac{\eps}{2} + 3^d L^{k,m_j} \bigg) \1_{A_j}.$$
Because ${\bar W}^{k,m_j}$ is ${\cal F}_{d-1,j}$-measurable and $B_j \in {\cal F}_{d-1,j}$, we can multiply both sides by ${\bar W}^{k,m_j} \1_{B_j}$ to get
\begin{align*}
E \big[ |W^{k,m_j} - {\bar W}^{k,m_j}| \1_{A_j \cap B_j} \big| {\cal F}_{d-1,j} \big] &\leq \bigg( \frac{\eps}{2} + 3^d L^{k,m_j}\bigg) {\bar W}^{k,m_j} \1_{A_j \cap B_j} \nonumber \\
&\leq \frac{3}{2} \bigg(\frac{\eps}{2} + 3^d L^{k,m_j} \bigg) W^{k,m_j,*,+}.
\end{align*}
Taking expectations of both sides, we get
$$E \big[ |W^{k,m_j} - {\bar W}^{k,m_j}| \1_{A_j \cap B_j} \big] \leq \frac{3 \eps}{4} + \frac{3^{d+1}}{2} E[W^{k,m_j,*,+} L^{k,m_j}].$$  The result now follows from Lemma \ref{Llem}.
\end{proof}

\begin{Lemma}\label{Wgamma}
Define $\mu^*$ as in Proposition \ref{fixedpt}, and let $\gamma = \int_2^{\infty} x \: \mu^*(dx)$ be the mean of $\mu^*$.  Suppose $s_j \gg m_j$ and $s_j n_j \gg m_j^2$.  Then $$\frac{1}{m_j} \sum_{k=1}^{m_j} W^{k,m_j} \rightarrow_{\mathrm{p}} \gamma,$$ where $\rightarrow_{\mathrm{p}}$ denotes convergence in probability as $j \rightarrow \infty$.
\end{Lemma}

\begin{proof}
Let $\eta > 0$, and let $\eps = \eta^2$.  By Proposition \ref{fixedpt}, we can choose a positive integer $d$ sufficiently large that the mean of $T^d(\delta_2)$ is greater than $\gamma - \eta$, and the mean of $T^d(\delta_{\infty})$ is less than $\gamma + \eta$.  Because the random variables $W^{k,m_j,*,+}$ are independent of one another and have the distribution $T^d(\delta_{\infty})$, and the random variables $W^{k,m_j,*,-}$ are independent of one another and have distribution $T^d(\delta_2)$, it follows from the Law of Large Numbers and the assumption that $\lim_{j\rightarrow \infty} m_j = \infty$ that
$$\lim_{j \rightarrow \infty} P \bigg(\gamma - \eta < \frac{1}{m_j} \sum_{k=1}^{m_j} W^{k,m_j,*,-} \leq \frac{1}{m_j} \sum_{k=1}^{m_j} W^{k,m_j,*,+} < \gamma + \eta \bigg) = 1.$$
Therefore, by (\ref{Wcompare}),
$$\lim_{j \rightarrow \infty} P \bigg( \gamma - \eta < \frac{1}{m_j} \sum_{k=1}^{m_j} W^{k,m_j,*} < \gamma + \eta \bigg) = 1.$$
It now follows from Lemma \ref{starbar} that
\begin{equation}\label{Wbarlim}
\lim_{j \rightarrow \infty} P \bigg( \gamma - \eta < \frac{1}{m_j} \sum_{k=1}^{m_j} {\bar W}^{k,m_j} < \gamma + \eta \bigg) = 1.
\end{equation}
Define $t_0$ and the events $A_j$ and $B_j$ as in Lemma \ref{WWstar}.
We have
\begin{align}\label{W3tms}
P \bigg( \bigg| \frac{1}{m_j} \sum_{k=1}^{m_j} W^{k,m_j} - \gamma \bigg| \geq 2 \eta \bigg) &\leq P(A_j^c) + P(B_j^c) + P \bigg( \bigg| \frac{1}{m_j} \sum_{k=1}^{m_j} {\bar W}^{k,m_j} - \gamma \bigg| \geq \eta \bigg) \nonumber \\
&\hspace{.2in} + P \bigg( \bigg| \frac{1}{m_j} \sum_{k=1}^{m_j} W^{k,m_j} - \frac{1}{m_j} \sum_{k=1}^{m_j} {\bar W}^{k,m_j} \bigg| \1_{A_j \cap B_j} \geq \eta \bigg).
\end{align}
Note that $\lim_{j \rightarrow \infty} P(A_j^c) = 0$ by (\ref{cdi}) and $\lim_{j \rightarrow \infty} P(B_j^c) = 0$ by Lemma \ref{starbar}.  The third term on the right-hand side of (\ref{W3tms}) tends to zero as $j \rightarrow \infty$ by (\ref{Wbarlim}).  By Lemma \ref{WWstar} and Markov's Inequality, for sufficiently large $j$ we have
$$P \bigg( \bigg| \frac{1}{m_j} \sum_{k=1}^{m_j} W^{k,m_j} - \frac{1}{m_j} \sum_{k=1}^{m_j} {\bar W}^{k,m_j} \bigg| \1_{A_j \cap B_j} \geq \eta \bigg) \leq \frac{1}{\eta m_j} \sum_{k=1}^{m_j} E \big[ |W^{k,m_j} - {\bar W}^{k,m_j}| \1_{A_j \cap B_j} \big] \leq \frac{\eps}{\eta} = \eta.$$
Because $\eta > 0$ is arbitrary, the result follows.
\end{proof}

\subsection{Proof of Theorem 1}

\begin{proof}[Proof of Theorem \ref{Speed}]
For positive integers $j$, let $m_j^+ = 2(1 + \eps)/ct_j$ and $m_j^- = 2(1 - \eps)/ct_j$.  It follows from (\ref{cdi}) that $m \tau_m \rightarrow 2/c$ as $m \rightarrow \infty$, which implies that almost surely $\tau_{m_j^+} \leq t_j \leq \tau_{m_j^-}$ for sufficiently large $j$.  Therefore, almost surely
\begin{equation}\label{Nsqueeze}
N(\tau_{m_j^-}) \leq N(t_j) \leq N(\tau_{m_j^+})
\end{equation}
for sufficiently large $j$.  The assumptions of Theorem \ref{Speed} imply that $s_j \gg m_j^+$ and $n_j s_j \gg (m^+_j)^2$ and the same is true for $m_j^-$.  Therefore, by Lemma \ref{Wgamma}, using $\rightarrow_{\mathrm{p}}$ to denote convergence in probability as $j \rightarrow \infty$, we have
$$\frac{\tau_{m_j^+-1} N(\tau_{m_j^+-1})}{m_j^+} = \frac{\tau_{m_j^+-1}}{m_j^+} \sum_{k=1}^{m_j^+} N^{k,m_j^+} = \frac{1}{m_j^+} \sum_{k=1}^{m_j^+} W^{k,m_j^+} \rightarrow_{\mathrm{p}} \gamma.$$
Now using again that $m \tau_m \rightarrow 2/c$ almost surely as $m \rightarrow \infty$, we get
$$(m_j^+)^{-2} N(\tau_{m_j^+}) \rightarrow_{\mathrm{p}} \frac{c \gamma}{2}$$ and therefore 
\begin{equation}\label{final+}
t_j^2 N(\tau_{m_j^+}) \rightarrow_{\mathrm{p}} \frac{2(1 + \eps)^2 \gamma}{c}.
\end{equation}
By the same reasoning,
\begin{equation}\label{final-}
t_j^2 N(\tau_{m_j^-}) \rightarrow_{\mathrm{p}} \frac{2(1 - \eps)^2 \gamma}{c}.
\end{equation}
By letting $\eps \rightarrow 0$, we obtain the result from (\ref{Nsqueeze}), (\ref{final+}), and (\ref{final-}).
\end{proof}

\section*{Acknowledgements}
This project began while the authors were attending a Bath, UNAM, and CIMAT (BUC) workshop in Guanajuato, Mexico in May, 2016.  The authors thank Andreas Kyprianou, Juan Carlos Pardo, and Victor Rivero for their roles in organizing this workshop. ABB is supported by CONACyT-MEXICO, TR is supported by the Royal Society, and JS is supported in part by NSF Grants DMS-1206195 and DMS-1707953.

\end{document}